\documentclass[]{siamart220329}
\usepackage[utf8]{inputenc}

\usepackage{enumitem}

\usepackage{tikz}
\usetikzlibrary{positioning}
\usetikzlibrary{arrows.meta}

\usepackage{amsmath}
\usepackage{mathtools}
\usepackage{amssymb}
\usepackage{amsfonts}
\usepackage{xcolor}

\newcommand{\natNum}{\mathbb{N}}
\newcommand{\realNum}{\mathbb{R}}
\newcommand{\complNum}{\mathbb{C}}

\newcommand{\dom}{\mathrm{D}}

\newcommand{\posRealNum}{[0,\infty)}

\newcommand{\semiGroup}[1]{\mathbb{T}(#1)}
\newcommand{\semiGroupDef}{\mathbb{T} = \left( \semiGroup{t} \right)_{t \geq 0}}
\newcommand{\semiGroupDefAlt}[1]{#1 = \left( #1 (t) \right)_{t \geq 0}}

\newcommand{\loc}{\mathrm{loc}}
\newcommand{\andMath}{\textrm{and}}
\newcommand{\realPart}[1]{\mathrm{Re}\left( #1 \right)}

\newcommand{\set}[1]{\left\lbrace #1 \right\rbrace}

\newcommand{\innerProd}[2]{\left\langle #1 , #2 \right\rangle}
\newcommand{\boundedOp}[2]{\mathcal{L} \left( #1 , #2 \right)}
\newcommand{\boundedOpSelf}[1]{\mathcal{L} \left( #1  \right)}

\newcommand{\lpSpaceDT}[3]{L^{#1}\left(#2, #3\right)}
\newcommand{\lpSpacelocDT}[3]{L^{#1}_{\loc}\left(#2, #3\right)}

\newcommand{\laplaceTr}[1]{\mathcal{L}\left\lbrace #1 \right\rbrace}
\newcommand{\laplaceInvTr}[1]{\mathcal{L}^{-1}\left\lbrace #1 \right\rbrace}

\newcommand{\intd}[1]{\,\mathrm{d}#1}
\newcommand{\der}[2]{\frac{\mathrm{d}#1}{\mathrm{d}#2}}
\newcommand{\derOrd}[3]{\frac{\mathrm{d}^{#3}#1}{\mathrm{d}#2^{#3}}}

\newcommand{\resolvent}[2]{\left( #1 \mathbb{I} - #2 \right)^{-1}}

\newsiamremark{remark}{Remark} 

\usepackage{autonum}

\title{On BIBO stability of infinite-dimensional linear state-space systems\thanks{Submitted to the editors on April 3, 2023.\funding{Wierzba is supported by the Theme Team project ``Predictive Avatar Control and Feedback'' sponsored by the Faculty of Eletrical Engineering, Computer Science and Mathematics.}}}
\author{Felix L.~Schwenninger\thanks{Department of Applied Mathematics, University of Twente, P.O. Box 217, 7500 AE Enschede, The Netherlands (\email{f.l.schwenninger@utwente.nl}, \email{a.a.wierzba@utwente.nl}, \email{h.j.zwart@utwente.nl})} 
\and
Alexander A. Wierzba\footnotemark[2]
\and
Hans Zwart\footnotemark[2]{ }\thanks{Department of Mechanical Engineering, Eindhoven University of Technology, P.O. Box 513, 5600 MB Eindhoven, The Netherlands (\email{h.j.zwart@tue.nl})} 
}

\headers{On BIBO stability of infinite-dimensional linear state-space systems}{Felix L. Schwenninger, Alexander A. Wierzba, Hans Zwart}

\begin{document}

\maketitle

\begin{abstract}
In this paper we consider BIBO stability of systems described by infinite-dimensional linear state-space representations, filling the so far unattended gap of a formal definition and characterization of BIBO stability in this general case.  
Furthermore, we provide several sufficient conditions guaranteeing BIBO stability of a particular system and discuss to which extent this property is preserved under additive and multiplicative perturbations of the system. 
\end{abstract}
%


\section{Introduction}
The concept of stability is omnipresent in control theory. In this paper we consider systems that feature both a control input $u:\posRealNum \rightarrow U$ and a corresponding observation output $y:\posRealNum \rightarrow Y$ mapping to normed space $U$ and $Y$, respectively, and study one particular notion of stability, namely \emph{bounded-input-bounded-output} stability, usually referred to by its abbreviation as \emph{BIBO stability}.

Roughly speaking, a system is BIBO stable if for any time interval $[0,T]$ and any control input $u$, we have that the corresponding output function $y$ is bounded uniformly relative to $u$, both with respect to the (essential) supremum norm. For a linear system this reduces to the existence of a constant $c > 0$ such that 
\begin{equation}
     \sup_{t\in[0,T]} \|y(t)\|_Y \leq c \sup_{t\in[0,T]}\|u(t) \|_U,
\end{equation}
for any $T >0$ and any $u$ in (a subspace of) $L^{\infty}([0,T];U)$.
In the following we focus on systems given in a state space representation, formally written as equations 
\begin{equation}
\label{eq:formalStateSpaceEquations}
\begin{array}{ccc}
\dot{x}(t)&=&Ax(t)+Bu(t)\\
y(t)&=&Cx(t)+Du(t)\end{array}\quad t\ge0,
\end{equation}
where $x(t) \in X$, $u(t) \in U$, $y(t) \in Y$ with Banach spaces $X$, $U$ and $Y$ and where $A$, $B$, $C$, $D$ refer to suitably defined linear operators. 
For finite-dimensional systems, that is, when $X$, $Y$ and $U$ are finite-dimensional normed spaces, BIBO stability reduces to the property that the impulse response $t\mapsto C\mathrm{e}^{tA}B$ is absolutely integrable on $\posRealNum$, which particularly holds true if the system is exponentially stable.

This becomes more subtle when considering infinite-dimensional state spaces which e.g.\ appear in the study of distributed parameter systems described by partial differential equations. While in the case of bounded operators $B:U\to X$ and $C:X\to Y$, i.e.\ distributed control and observation, the property that $t \mapsto C e^{t A} B$ is absolutely integrable still characterises BIBO stability, this changes in general.
In particular this challenge emerges when modelling boundary control and observation by state-space representations  \eqref{eq:formalStateSpaceEquations}. This give rise to more involved issues regarding the existence  of solutions and the well-definedness of concepts such as transfer function, impulse response and input-output mapping -- subtleties, which do not arise in the finite-dimensional case. One remarkable consequence of this setting is, for example, that exponential stability of the system does no longer imply BIBO stability (see e.g.\ the system in Theorem~\ref{thm:MultPertResult}).

While we show in the following that also in the infinite-dimensional case a characterisation of BIBO stability in terms of the transfer function exists, establishing it in a concrete case may become prohibitively hard, hence calling for alternative approaches. Note in particular, that the conditions to be checked are significantly more involved than, for example, those in the case of $L^2$-to-$L^2$ input-output stability, for which only boundedness of the transfer function needs to be checked \cite{a_TucsnakWeissWellPosed2014}. 

It is worth noting that BIBO stability is closely related to other system-theoretic stability notions. It places itself within the wider class of input-output-stabilities, all related to the property that boundedness of the output can be inferred from boundness of the input, both with respect to some function norms.

Stabilities of such form in terms of the $L^2$-norms for both input and output functions, have been studied extensively within the context of $L^2$-well-posed linear systems and their generalisation to other $L^p$-spaces. However usually intertwined with additional input-state stability considerations \cite{b_Staffans2005}. Thus, BIBO stability is similar to, but more general than the concept of $L^{\infty}$-well-posed systems \cite[Def. 2.2.1]{b_Staffans2005}, as in the latter case not only boundedness of the input-output map but also of the input-state and state-output maps with respect to the  $\| \cdot \|_{\infty}$-norms is required. Therefore BIBO stability seems  to be more difficult to study.

On the other hand, \emph{input-state-stability} (ISS) refers to bounding the state function in the $\| \cdot \|_{\infty}$-norm by a combination of the input in the $\| \cdot \|_{\infty}$-norm and the state space norm of the initial state \cite{a_SontagISS1989,b_SontagISS2008}. The state-output mapping included in this notion again is not considered in the case of BIBO stability. Note however, that in the case that the ouput operator $C: X \rightarrow Y$ is bounded, ISS and BIBO stability are -- although very different and in particular not equivalent -- of the same complexity.

From an operator-theoretical viewpoint, BIBO stability can be phrased as the question of whether the input-output operator of the system is bounded as an operator from $L^\infty([0,t],U)$ to $L^\infty([0,t],Y)$ with its norm not depending on $t$. Operators of this type -- bounded from some $L^p$ to $L^p$ and in addition shift-invariant (as is the case for the input-output operators of systems formally of the form \eqref{eq:formalStateSpaceEquations}) -- are closely related to operators represented by convolutions with distributions of certain classes. However, in the case considered for BIBO stability, i.e.\ $p = \infty$, this relation is more subtle compared to $p < \infty$, in which case one even finds a one-to-one correspondence provided that $U$ and $Y$ are finite-dimensional \cite{b_Grafakos2014}. 

BIBO stability has been variously employed in control theory, most recently for example in the context of novel control techniques, such as e.g.\ funnel control for systems of relative degree, \cite{a_IlchmannRzanSangwinFunnel2002,a_BergerIlchmannRyanFunnel2021}. In fact, here BIBO stability of the, possibly infinite-dimensional, internal dynamics is key to apply the model-independent control law, see e.g.\  \cite{a_IlchmannRyanTrennFunnel2005,a_BergerPucheSchwenningerFunnel2020,a_WaterTank2022}, also for concrete control problems.

While BIBO stability of finite-dimensional systems is well-studied in the literature, particularly in the context of engineering applications, there is only little treatment available in infinite dimensions. 
This includes systems given by convolutions \cite{a_UnserNoteOnBiboStability, a_WangCobb96BIBOTimeInv, a_WangCobb03BIBOTimeVar}, specific classes of transfer functions \cite{a_CallierDesoer1978, a_AbusaksakaPartingtonBIBO2013} and concrete PDEs  \cite{a_WaterTank2022}. 

In this paper, we first consider several different definitions of BIBO stability for general system nodes, distinguishing between different functions spaces equipped with the $\|\cdot\|_\infty$-norm, Section~\ref{sec:SystemNodes}. We then show that for a large class of system nodes, namely when both the input and the output space are finite-dimensional, these different definitions will all be equivalent to the inverse Laplace transform of the transfer function being a measure of bounded variation, Section~\ref{sec:BIBOStability}.
This result may seem, at first glance, to capture the same statement as presented by Unser in \cite{a_UnserNoteOnBiboStability} and in particular Theorem 4 therein. A closer look, however, reveals major distinctions. Indeed, this work extends the results from \cite{a_UnserNoteOnBiboStability} to systems in state-space form which are not a-priori given in terms of a convolution operator. For a more detailed discussion of the relation between these results, we refer the reader to Section~\ref{sec:RelationToUngerResult}.

To overcome the technical difficulties of studying the inverse Laplace transform of a given function, in Section~\ref{sec:SufficientConditions} we will provide three simple sufficient conditions for particular systems to be BIBO stable.

Finally, we consider the case that the system investigated can be written as a perturbation of a BIBO stable system. Such a situation can appear e.g.\ when additional terms are introduced in the system modelling.
Whereas BIBO stability of multiplicative perturbations may fail, Section~\ref{sec:MultPerturbationsChapter}, we show that it is preserved in certain parabolic cases for additive perturbations, Section~\ref{sec:AddPerturbationsChapter}.

\subsection{Notation}
We will use $A \lesssim B$ to indicate that there exists a constant $c > 0$ independent of the other quantified variables, such that $A \leq c B$.

For a function $z:\posRealNum \rightarrow Z$ and $T>0$, we denote by $z|_{[0,T]}:\posRealNum \rightarrow Z$ its truncation to the interval $[0,T]$ defined by $z|_{[0,T]}(t):=z(t)\chi_{[0,T]}(t)$, where $\chi_{A}$ refers to the indicator function on the set $A$.
We will use $C^n_c\left( (0, T), Z \right)$ to denote the space of $n$-times continuously differentiable $Z$-valued functions $z:(0,T) \rightarrow Z$  that are compactly supported away from $0$.

For any $\alpha \in \realNum$, let $\complNum_\alpha := \left\lbrace z \in \complNum \;\middle|\; \realPart{z} > \alpha \right\rbrace$. All considered Banach spaces are over the field of complex numbers. For a finite-dimensional vector space $Z$, let its dimension be denoted by $d_Z$. Let $\boundedOp{X}{Y}$ be the set of bounded operators from a Banach space $X$ to a Banach space $Y$. $B^*$ is the Hilbert space adjoint of a linear operator $B$ between two Hilbert spaces.

In the following $A$ will always denote the generator of a strongly continuous semigroup $\semiGroupDef$ on a Banach space $X$.  Its growth bound is denoted by $\omega(\mathbb{T})=\inf\{\omega\in\realNum\colon\sup_{t>0}e^{-t\omega}\|\mathbb{T}(t)\|<\infty\}$ and the semigroup is called exponentially stable if $\omega(\mathbb{T})<0$. The semigroup is called (bounded) analytic if it extends (boundedly and) analytically to an open sector $\{z\in\complNum\setminus\{0\}\colon |\arg(z)|<\theta\}$ for some $\theta\in(0,\pi/2)$ and this extension satisfies the semigroup property on this sector.
$X_1$ is defined as the space $\dom(A)$ with the norm $\|x\|_{X_1} := \| (\beta I - A) x \|_X$ with $\beta \in \rho(A)$ and $X_{-1}$ as the completion of $X$ with respect to the norm $\|x\|_{X_{-1}} := \| (\beta I - A)^{-1} x \|_X$ again with $\beta \in \rho(A)$. There exists a unique extension $\mathbb{T}_{-1} = \left( \mathbb{T}_{-1}(t) \right)_{t \geq 0}$ of the semigroup $\mathbb{T}$ to the space $X_{-1}$ with generator $A_{-1}:X \rightarrow X_{-1}$ which is an extension of the operator $A$. For more details on these notions see \cite{b_Staffans2005,b_TucsnakWeiss2009}. In the following we will often not write the subscript ${}_{-1}$ and refer to both the original semigroup and its extension by $\mathbb{T}$. The same applies to the operators $A$ and $A_{-1}$.
If $A$ is the generator of a bounded analytic semigroup, let $X_\alpha$ with $-1 < \alpha < 1$ denote the domains of definition of the fractional powers of $(-A)^{\alpha}$ equipped with the respective norms as defined e.g.\ in \cite[Sec.~3.9]{b_Staffans2005}.

For a function $f$ or a measure $h$, let $\hat{f}$ and $\laplaceTr{h}$ denote their Laplace transforms if they exist (on some right half-plane) and let $\laplaceInvTr{g}$ denote the inverse Laplace transform of $g$ if it exists.

\section{System nodes and their solutions}
\label{sec:SystemNodes}
We begin by giving the definition of general system nodes in the sense of Staffans as the description of systems which are formally given by the pair of equations~\eqref{eq:formalStateSpaceEquations}.
\begin{definition}[{\cite[Def.~4.1]{a_TucsnakWeissWellPosed2014}}]
Let $U$, $X$ and $Y$ be Banach spaces. A \emph{system node} on $(U,X,Y)$ is a collection $\Sigma(A,B,C,\mathbf{G})$ where $A: \dom(A) \subset X \rightarrow X$ is the generator of a strongly continuous semigroup $\semiGroupDef$ on $X$ with growth bound $\omega(\mathbb{T})$, $B \in \boundedOp{U}{X_{-1}}$, $C \in \boundedOp{X_1}{Y}$ and $\mathbf{G}: \complNum_{\omega(\mathbb{T})} \rightarrow \boundedOp{U}{Y}$ is an analytic function satisfying for all $\alpha, \beta \in \complNum_{\omega(\mathbb{T})}$
    \begin{equation}
    \label{eq:transferFunctionDifferenceRelation}
        \mathbf{G}(\alpha) - \mathbf{G}(\beta) = C\left[ \resolvent{\alpha}{A} - \resolvent{\beta}{A} \right] B.
    \end{equation}
\end{definition}

\begin{remark}
\label{rem:transferFunctionDifferConstant}
\begin{enumerate}
    \item We will call $U$, $X$ and $Y$ the \emph{input space}, \emph{state space} and \emph{output space}, respectively. Furthermore, $B$ is called the \emph{input operator}, $C$ the \emph{output operator} and $\mathbf{G}$ the \emph{transfer function} of the system node. 
    \item Equation~\eqref{eq:transferFunctionDifferenceRelation} does not uniquely determine the transfer function $\mathbf{G}$ from $A$, $B$ and $C$. It is instead easy to show that any two transfer functions that satisfy Equation ~\eqref{eq:transferFunctionDifferenceRelation} for the same operators $A$, $B$ and $C$ differ only by a constant in $\boundedOp{U}{Y}$  \cite{a_TucsnakWeissWellPosed2014}. Note however that these different transfer functions give rise to two different system nodes with differing input-output behaviour.
\end{enumerate}
\end{remark}

\begin{definition}[{\cite[Sec.~4]{a_TucsnakWeissWellPosed2014}}]
Let $\Sigma(A,B,C,\mathbf{G})$ be a system node. The \emph{combined observation/feedthrough operator} is defined as the operator
\begin{equation}
\label{eq:defCandDOperator}
    C\&D \begin{bmatrix} x \\ u \end{bmatrix} := C \left[ x - \resolvent{\beta}{A} B u \right] + \mathbf{G}(\beta) u,
\end{equation}
with $\beta \in \complNum_{\omega(\mathbb{T})}$ and domain
\begin{equation}
    \dom(C\&D) = \set{\begin{bmatrix} x \\ u \end{bmatrix} \in X \times U \middle| Ax + Bu \in X }.
\end{equation}
\end{definition}
Note that the operator $C\& D$ is bounded from its domain equipped with the graph norm $\left\| \left[ \begin{smallmatrix} x \\ u \end{smallmatrix} \right] \right\|_{\dom(C\&D)} = \| u \|_U + \|x\|_X + \| A_{-1} x + B u \|_X$ to $Y$ \cite[Lem.~4.7.3~\&~4.3.10]{b_Staffans2005}. With this operator the well-defined version of the formal expressions~\eqref{eq:formalStateSpaceEquations} associated to the system node $\Sigma(A,B,C,\mathbf{G})$ becomes
\begin{equation}
\label{eq:systemNodeDifferentialEqns}
    \begin{split}
        \Dot{x}(t) &= A x(t) + B u(t) \\
        y(t) &= C\&D \begin{bmatrix} x(t) \\ u(t) \end{bmatrix}.
    \end{split}
\end{equation}
In the following, we will consider two different types of solutions to these  equations, namely, the classical solutions and generalised solutions in a distributional sense.

\subsection{Classical solutions}
\begin{definition}[{\cite[Def.~4.2]{a_TucsnakWeissWellPosed2014}}]
Let $\Sigma(A,B,C,\mathbf{G})$ be a system node. A triple $(u,x,y)$ is called a \emph{classical solution} of the system node on $[0,\infty)$ if 
\begin{enumerate}[label=(\alph*)]
    \item $u \in C([0,\infty),U)$, $x \in C^1([0,\infty),X)$ and $y \in C([0,\infty),Y)$,
    \item $\begin{bmatrix} x(t) \\u(t) \end{bmatrix} \in \dom(C\&D) \quad \forall t \geq 0$,
    \item Equations~\eqref{eq:systemNodeDifferentialEqns} holds for all $t\geq 0$.
\end{enumerate}
\end{definition}
We cite the following results that guarantee the existence of such classical solutions and provide a representation of the state trajectory and the output function.
\begin{lemma}
\label{lem:existenceClassicalSolutions}
\label{lem:TransferFunctionMultiRepCompact}
Let $\Sigma(A,B,C,\mathbf{G})$ be a system node.
\begin{enumerate}
    \item \label{enum:ExistenceClassSol} Let $u \in C^2 ([0,\infty), U)$ and $\begin{bmatrix} x_0 \\ u(0) \end{bmatrix} \in \dom(C\&D)$. Then  Equation~\eqref{eq:systemNodeDifferentialEqns} has a unique classical solution $(u,x,y)$ with $x(0) = x_0$ and $x \in C^2([0,\infty),X_{-1})$.
    \item Let $(u,x,y)$ be a classical solution of $\Sigma(A,B,C,\mathbf{G})$. Then:
    \begin{enumerate}
        \item \label{enum:VarConstFormula} $x$ satisfies the \emph{variations of constants formula}, i.e.\
        \begin{equation}
            \label{eq:varOfConstantsFormula}
            x(t) = \semiGroup{t} x(0) + \int_0^t \mathbb{T}_{-1}\left( t - s \right) B u(s) \intd{s}.
        \end{equation}
        \item \label{enum:ComSuppLaplaceTrafo} If $u \in C^2_c([0,\infty), U)$, then $y$ has a Laplace transform on $\complNum_{\omega(\mathbb{T})}$ and
        \begin{equation}
            \hat{y} = G \cdot \hat{u} \qquad \textrm{on } \complNum_{\omega(\mathbb{T})}.
        \end{equation}
    \end{enumerate}
\end{enumerate}
\end{lemma}
\begin{proof}
Part~\ref{enum:ExistenceClassSol} is \cite[Lem.~4.7.8]{b_Staffans2005} and \cite[Prop.~4.3]{a_TucsnakWeissWellPosed2014}, Part~\ref{enum:VarConstFormula} is \cite[Thm.~4.7.11]{b_Staffans2005} and Part~\ref{enum:ComSuppLaplaceTrafo} is \cite[Lem.~4.7.12]{b_Staffans2005} and \cite[Prop.~4.3]{a_TucsnakWeissWellPosed2014}.
\end{proof}

\subsection{Distributional solutions in the sense of Staffans}
Motivated by the variations of constants formula~\eqref{eq:varOfConstantsFormula} we define generalised solutions in a distributional sense. Thereby we make use of  $Y$-valued distributions, which are linear functionals on the space of test functions $C^{\infty}_c(\posRealNum, Y')$. We denote the action of a distribution $y$ on a test function $\varphi \in C^{\infty}_c(\posRealNum, Y')$ by $y\left( \varphi \right)$. Note that any function $y \in \lpSpacelocDT{1}{\posRealNum}{Y}$ can be considered as a $Y$-valued distribution acting 
via
\begin{equation}
    y\left( \varphi \right) = \langle \varphi, y \rangle := \int_0^\infty \left\langle \varphi(t) , y(t) \right\rangle_{Y', Y} \intd{t}.
\end{equation}
\begin{definition}[{\cite[Def.~4.7.10]{b_Staffans2005}}]
\label{def:distributionalSolution}
A triple $(u,x,y)$ is called a \emph{generalised solution in the distributional sense} of the system node $\Sigma(A,B,C,\mathbf{G})$ if
\begin{enumerate}[label=(\alph*)]
    \item $u \in L^1_{\loc}([0,\infty), U)$ and $x \in C([0,\infty), X_{-1})$
    \item $x(t) = \semiGroup{t} x_0 + \int_0^t \mathbb{T}_{-1}\left( t - s \right) B u(s) \intd{s}$ for some $x_0 \in X$ and all $t \geq 0$
    \item $y$ is the $Y$-valued distribution given distributionally as 
    \begin{equation}
    y(t) = \derOrd{}{t}{2} \left( \left( C\&D \right) \int_0^t (t - s) \begin{bmatrix} x(s) \\ u(s) \end{bmatrix} \intd{s} \right), \quad t \geq 0    
    \end{equation}
    that means which acts on test functions $\varphi \in C^{\infty}_c(\posRealNum, Y')$ as
    \begin{equation}
    \begin{split}
        y\left( \varphi \right) &= \left\langle \varphi'', \left( C\&D \right) \int_0^\cdot (\cdot - s) \begin{bmatrix} x(s) \\ u(s) \end{bmatrix} \intd{s} \right\rangle \\ &= \int_0^\infty \left\langle \varphi''(t) , \left( C\&D \right) \int_0^t (t - s) \begin{bmatrix} x(s) \\ u(s) \end{bmatrix} \intd{s} \right\rangle_{Y', Y} \intd{t}.
    \end{split}
    \end{equation}
\end{enumerate}
\end{definition}
By the following lemma generalised solutions exist for any input $u \in L^1_{\loc}([0,\infty), U)$.
\begin{lemma}[{\cite[Thm.~3.8.2(i)~\&~Lem.~4.7.9]{b_Staffans2005}}]
\label{lem:existenceUniquenessGenSolutions}
For any $u \in L^1_{\loc}([0,\infty), U)$ and $x_0 \in X$ there exists a unique generalised solution in the distributional sense  of $\Sigma(A,B,C,\mathbf{G})$.

Furthermore, any classical solution $(u,x,y)$ is a generalized solution in the distributional sense.
\end{lemma}
\section{BIBO stability}
\label{sec:BIBOStability}
Having introduced the concept of system nodes and their solutions, we can now define the notion of BIBO stability. This can be done in two different ways, either using classical solutions, or using the class of distributional solutions. It will turn out that for the case of finite-dimensional input and output spaces $U$ and $Y$ these two definitions are equivalent. 

\begin{definition}[Using distributional solutions]
\label{def:BIBOStabAllSolutions}
A system node $\Sigma(A,B,C,\mathbf{G})$ is called $L^\infty$-\emph{BIBO stable} if there exists $c > 0$ such that for any generalised solution in the distributional sense $(u,x,y)$ with $u \in L^\infty_\loc(\posRealNum, U)$ and $x(0) = 0$ we have that $y \in L^\infty_\loc(\posRealNum, Y)$ and for all $t > 0$
\begin{equation}
\label{eq:BIBOInequality}
    \| y \|_{\lpSpaceDT{\infty}{[0,t]}{Y}} \leq c \| u \|_{\lpSpaceDT{\infty}{[0,t]}{U}}.
\end{equation}
\end{definition}

\begin{definition}[Using classical solutions]
\label{def:BIBOStabClassicalSolutions}
The system node $\Sigma(A,B,C,\mathbf{G})$ is called $C^\infty$-\emph{BIBO stable} if there exists $c > 0$ such that for all classical solutions $(u,x,y)$ with $x(0) = 0$ and $u \in C_c^\infty((0,\infty),U)$ we have that $y \in C_b(\posRealNum, Y)$ and
\begin{equation}
    \| y \|_{\lpSpaceDT{\infty}{[0,\infty)}{Y}} \leq c \| u \|_{\lpSpaceDT{\infty}{[0,\infty)}{U}}.
\end{equation}
\end{definition}
\begin{remark}
    In the latter definition the choice of $C_c^\infty((0,\infty),U)$ may seem arbitrary. Indeed one could have chosen to use another subspace $K \subseteq C(\posRealNum, U)$ to define a notion of $K$-BIBO stability.
    
    However, Theorems~\ref{thm:classicalBIBOBoundedVariation} and \ref{thm:generalBIBOclassicalBIBO} show that if the input and output spaces are finite-dimensional, all such definitions are equivalent as long as $C_c^\infty((0,\infty),U) \subseteq K$.
    
\end{remark}
For the remainder of this section, we will restrict ourselves to the case that both the input space $U$ and the output space $Y$ are finite dimensional. This restriction will be used to characterise the input-output map by convolution operators, which seems to be no longer possible in the more general case.

The following two equivalences will be proved in the following sections.
\begin{theorem}
\label{thm:classicalBIBOBoundedVariation}
A system node $\Sigma(A,B,C,\mathbf{G})$ with finite-dimensional input and output spaces is $C^\infty$-BIBO stable if and only if there exists a measure of bounded total variation $h \in \mathcal{M}(\posRealNum, \complNum^{d_Y \times d_U})$ such that $\laplaceTr{h} = \mathbf{G}$ on the half-plane $\complNum_{\omega(\mathbb{T})}$.
\end{theorem}
\begin{proof}
Necessity and sufficiency are Proposition~\ref{prop:boundedVarImpliesBIBO} and Proposition~\ref{lem:transferFunctLaplaceMeasure} below.
\end{proof}
\begin{theorem}
\label{thm:generalBIBOclassicalBIBO}
A system node $\Sigma(A,B,C,\mathbf{G})$ with finite-dimensional input and output spaces is $C^\infty$-BIBO stable  if and only if it is $L^\infty$-BIBO stable.
\end{theorem}
\begin{proof}
Sufficiency follows from Lemma~\ref{lem:existenceUniquenessGenSolutions}, necessity is Proposition~\ref{prop:BIBOIneqForDistrSol} below.
\end{proof}

\subsection{Proof of Theorem~\ref{thm:classicalBIBOBoundedVariation}}

\subsubsection{From a measure of bounded variation to $C^\infty$-BIBO stability}
Let $\mathcal{M}(\posRealNum, \complNum^{d_Y \times d_U})$ denote the set of Borel measures of bounded total variation on $\posRealNum$ with values in $\complNum^{d_Y \times d_U}$. 
Furthermore, for $h \in \mathcal{M}(\posRealNum, \complNum^{d_Y \times d_U})$ let $\| h \|_{\mathcal{M}}$ denote the total variation of $h$ \cite[Sec.~3.2]{b_GripenbergLondenStaffans1990}. Then there is a simple, sufficient condition for a system node to be BIBO stable.
\begin{proposition}
\label{prop:boundedVarImpliesBIBO}
Let $\Sigma(A,B,C,\mathbf{G})$ be a system node. Assume there exists a measure of bounded total variation $h \in \mathcal{M}(\posRealNum, \complNum^{d_Y \times d_U})$ such that $\laplaceTr{h} = \mathbf{G}$ on the half-plane $\complNum_{\omega(\mathbb{T})}$. Then $\Sigma(A,B,C,\mathbf{G})$ is $C^\infty$-BIBO stable.
\end{proposition}
\begin{proof}
Let $u \in C_c^\infty((0,\infty),U)$ and $y\in C([0,\infty),Y)$ be the output of the classical solution with $x(0) = 0$. By Lemma~\ref{lem:TransferFunctionMultiRepCompact} we have that $\hat{y} = \mathbf{G} \cdot \hat{u}$. 
As the inverse Laplace transform of $\mathbf{G}$ exists by assumption, this implies by \cite[Thm.~3.8.1(iii)]{b_GripenbergLondenStaffans1990} that $y = h \ast u$ and thus by \cite[Thm.~3.6.1(i)]{b_GripenbergLondenStaffans1990} that $y \in  C_b(\posRealNum, Y)$ and that 
\begin{equation}
    \|y\|_{\lpSpaceDT{\infty}{\posRealNum}{Y}} \leq \| h \|_{\mathcal{M}} \|u\|_{\lpSpaceDT{\infty}{\posRealNum}{U}}.
\end{equation}
\end{proof}

\begin{remark}
    In general, $\laplaceTr{h}$ for $h \in \mathcal{M}(\posRealNum, \complNum^{d_Y \times d_U})$ only exists on $\complNum_0$ and thus $\laplaceTr{h}=\mathbf{G}$ only holds on $\complNum_0 \cap \complNum_{\omega(\mathbb{T})}$. However, then $\laplaceTr{h}$ will have an analytic continuation to $\complNum_{\omega(\mathbb{T})}$, which in the following we will  always identify with $\laplaceTr{h}$.
\end{remark}

\subsubsection{From $C^\infty$-BIBO stability to a measure of bounded variation}

\begin{proposition}
\label{prop:existenceConvOperator}
\label{lem:transferFunctLaplaceMeasure}
Let $\Sigma(A,B,C,\mathbf{G})$ with finite-dimensional input and output spaces be $C^\infty$-BIBO stable. Then there exists a measure of bounded total variation $h \in \mathcal{M}(\posRealNum, \complNum^{d_Y \times d_U})$ such that:
\begin{enumerate}
    \item For any classical solution $(u,x,y)$ with $u \in  C_c^\infty((0,\infty),U) $ and $x(0) = 0$ we have $ y = h \ast u$.
\item $\laplaceTr{h} = \mathbf{G}$ on the half-plane $\complNum_{\omega(\mathbb{T})}$.
\end{enumerate}
\end{proposition}
\begin{proof}
Let $S:  C_c^\infty((0,\infty),U) \rightarrow C([0,\infty), Y)$ be the  \emph{input-output} operator of the system node that maps the input function $u$ to the corresponding output $y$ of the classical solution with $x(0) = 0$. By Lemma~\ref{lem:existenceClassicalSolutions} this operator is well-defined.

For $T \in \realNum$, let $\tau_T$ be the shift-operator, defined on a function $f: \realNum \rightarrow U$ as
\begin{equation}
    (\tau_T f)(t) := f(t - T), \qquad t \in \realNum.
\end{equation}
Considering a function on $\posRealNum$ as defined on all of $\realNum$ by trivially extending it with $0$, we can directly see that for all $T\geq0$ (i.e.\ for right-shifts) $\tau_T$ leaves $C_c^\infty((0,\infty),U)$ invariant, i.e.\ $\tau_T \left( C_c^\infty((0,\infty),U) \right) \subseteq C_c^\infty((0,\infty),U)$. Furthermore, from causality and the uniqueness of classical solutions of the system node, it follows that $S$ commutes with $\tau_T$ for all $T \geq 0$, i.e.\ $\tau_T S = S \tau_T$. 

Now define the operator $\widetilde{S}: C^\infty_c(\realNum, U) \rightarrow C(\realNum, Y)$ by setting for $u \in C^\infty_c(\realNum, U)$ 
\begin{equation}
    \widetilde{S} u := \tau_{-T} \, S \, \tau_{T} \, u,
\end{equation}
where $T \geq 0$ is such that $u(t) = 0$ for all $t < -T$. The compact support ensures the existence of such $T$ and $S$ commuting with right-shifts ensures that  $\widetilde{S} u$ does not depend on the chosen $T$.

Then $\widetilde{S}$ is clearly linear, invariant under left and right shifts and, as a result of the assumed BIBO stability of the system node, 
\begin{equation}
    \| \widetilde{S} u \|_{\lpSpaceDT{\infty}{\realNum}{Y}} \lesssim \| u \|_{\lpSpaceDT{\infty}{\realNum}{U}}. 
\end{equation}
Then by applying \cite[Thm.~I.3.16, Thm.~I.3.19 \& Thm.~I.3.20]{b_SteinWeiss2016} componentwise, we find that there exists $h \in \mathcal{M}(\realNum, \complNum^{d_Y \times d_U})$ such that
\begin{equation}
    \widetilde{S} u = h \ast u \qquad \textrm{for all } u \in C^\infty_c(\realNum, U).
\end{equation}
By the construction of $\widetilde{S}$ we have that $S = \widetilde{S}|_{C_c^\infty((0,\infty),U) }$ and therefore
\begin{equation}
\label{eq:convolRepresentationSmoothSols}
    S u = h \ast u \qquad \textrm{for all } u \in C_c^\infty((0,\infty),U).
\end{equation}
Finally, the causality of the system node implies that $h \in \mathcal{M}(\posRealNum, \complNum^{d_Y \times d_U})$.

For the second part, note that, as $h \in \mathcal{M}(\posRealNum, \complNum^{d_Y \times d_U})$, the Laplace transform $\laplaceTr{h}$ exists  and  from Equation~\eqref{eq:convolRepresentationSmoothSols} it thus follows that for all classical solutions $(u,x,y)$ with $u \in  C_c^\infty((0,\infty),U)$, we have $\hat{y} = \laplaceTr{h} \cdot \hat{u}$.
Comparing this with the result from Lemma~\ref{lem:TransferFunctionMultiRepCompact} yields the claim.
\end{proof}

\subsection{Proof of Theorem~\ref{thm:generalBIBOclassicalBIBO}}

\subsubsection{BIBO-inequality for all classical solutions with $x(0) = 0$}
By Definition~\ref{def:BIBOStabClassicalSolutions}, $C^\infty$-BIBO stability only requires the BIBO-inequality to hold for classical solutions $(u,x,y)$ with $u \in C_c^\infty((0,\infty),U)$ and $x(0) = 0$. As a first step we show that this  implies the same inequality for all classical solutions $(u,x,y)$ with $x(0) = 0$.

\begin{proposition}
\label{thm:BIBOInequalityAllClassSols}
Let $\Sigma(A,B,C,\mathbf{G})$  with finite-dimensional input and output spaces be $C^\infty$-BIBO stable. Then there exists a measure of bounded variation $h \in \mathcal{M}(\posRealNum, \complNum^{d_Y \times d_U})$ such that for any classical solution $(u,x,y)$ with $x(0) = 0$ we have that $y = h \ast u$ and thus that 
\begin{equation}
    \| y \|_{\lpSpaceDT{\infty}{[0,T]}{Y}} \leq \| h \|_{\mathcal{M}} \| u \|_{\lpSpaceDT{\infty}{[0,T]}{U}} \qquad \textrm{for all } T \geq 0.
\end{equation}
\end{proposition}
\begin{proof} 
Let $(u,x,y)$ be a classical solution of the system node with $x(0) = 0$.

Define for any $T > 0$ the smoothly cut-off input $u_T \in C_c(\posRealNum, U)$ as
\begin{equation}
    u_T(t) := \begin{cases} u(t) & 0 \leq t \leq T \\ u(T) \xi(t - T) & T < t \leq T+1 \\ 0 & T+1 < t \end{cases},
\end{equation}
where $\xi \in C^\infty([0,1], \realNum)$ is some smooth function satisfying $\xi(0) = 1$, $\xi(1) = \xi'(1) = \xi''(1) = 0$. Our aim is to show that there exists a classical  solution $(u_T,x_T,y_T)$ with this input.
For this, consider first the input function $w_{u(T)}: \posRealNum \rightarrow U$ defined as
\begin{equation}
    w_{u(T)}(t) := \begin{cases} u(T) \xi(t) & t \leq 1 \\ 0 & 1 < t \end{cases},
\end{equation}
for which clearly $w_{u(t)} \in C^2([0,\infty),U)$. Furthermore, as by \cite[Thm.~4.7.11]{b_Staffans2005}
\begin{equation}
 \begin{bmatrix} x(T) \\ w_{u(T)}(0) \end{bmatrix} = \begin{bmatrix} x(T) \\ u(T) \end{bmatrix} \in \dom(C \& D),   
\end{equation}
by \cite[Lem.~4.7.8]{b_Staffans2005} there exists a unique classical solution $(w_{u(T)}, x_w, y_w)$ of the system node with $x_w(0) = x(T)$.

Then define $x_T: \posRealNum \rightarrow X$ and $y_T: \posRealNum \rightarrow Y$ as
\begin{equation}
    x_T(t) := \begin{cases} x(t) & t \leq T \\ x_w(t - T) & T < t \end{cases} \quad \andMath \quad y_T(t) := \begin{cases} y(t) & t \leq T \\ y_w(t - T) & T < t \end{cases}.
\end{equation}
Then we have $u_T \in C(\posRealNum, U)$, $x_T \in C(\posRealNum, X)$, $x_T \in C^1([0,T], X)$ and $x_T \in C^1([T,\infty], X)$  by construction. Furthermore, as both $(u,x,y)$ and $(w_{u(T)}, x_w, y_w)$ are classical solutions of the system node, it follows from Equation~\eqref{eq:systemNodeDifferentialEqns}  that 
\begin{equation}
\begin{split}
    &\Dot{x}_w(0) = A x_w(0) + B w_{u(T)}(0) = A x(T) + B u(T) = \Dot{x}(T), \\
    &y_w(0) = C\&D \begin{bmatrix} x_w(0) \\ w_{u(T)}(0) \end{bmatrix} = C\&D \begin{bmatrix} x(T) \\ u(T) \end{bmatrix} = y(T),
\end{split}
\end{equation}
and thus we can conclude that $x_T \in C^1(\posRealNum, X)$ and $y_T \in C(\posRealNum, Y)$. Furthermore, Equation~\eqref{eq:systemNodeDifferentialEqns} is satisfied on $\posRealNum$. Hence $(u_T,x_T,y_T)$ is a classical solution of the system node.

We can now argue completely analogous to \cite[Lem.~4.7.12]{b_Staffans2005} to first conclude that $u_T$, $x_T$ and $y_T$ have Laplace transforms defined on $\complNum_{\omega(\mathbb{T})}$ and that they satisfy 
\begin{equation}
    \widehat{y_T}(s) = \mathbf{G}(s) \cdot \widehat{u_T}(s), \qquad s \in \complNum_{\omega_\mathbb{T}}.
\end{equation}
From the relation between $\widehat{y_T}$ and $\widehat{u_T}$, and $\mathbf{G}$ having an inverse Laplace transform $h \in \mathcal{M}(\posRealNum, \complNum^{d_Y \times d_U})$ by Proposition~\ref{lem:transferFunctLaplaceMeasure}, we then find that $y_T = h \ast u_T$.

Finally, by causality of the system node and the uniqueness of classical solutions we have for all $T \geq 0$ that $y|_{[0,T]} = y_T|_{[0,T]}$.
Thus for any $t \geq 0$ we find in particular
\begin{equation}
    y(t) = y|_{[0,t]}(t) = y_t|_{[0,t]}(t) = (h \ast u_t)(t) = (h \ast u)(t),
\end{equation}
and thus can conclude that $y = h \ast u$.
The inequality follows then immediately from this representation and $h \in \mathcal{M}(\posRealNum, \complNum^{d_Y \times d_U})$.
\end{proof}

\subsubsection{BIBO-inequality for generalised distributional solutions}
Proposition~\ref{thm:BIBOInequalityAllClassSols} shows that in a $C^\infty$-BIBO stable system we have for any classical solution $(u,x,y)$ with $x(0) = 0$ and any $T\geq0$ that
    $\| y \|_{\lpSpaceDT{\infty}{[0,T]}{Y}} \leq \| h \|_{\mathcal{M}} \| u \|_{\lpSpaceDT{\infty}{[0,T]}{U}}$. 
We extend this inequality to all distributional solutions with $u \in \lpSpacelocDT{\infty}{\posRealNum}{U}$.
\begin{proposition}
\label{prop:BIBOIneqForDistrSol}
Let $\Sigma(A,B,C,\mathbf{G})$ with finite-dimensional input and output spaces be $C^\infty$-BIBO stable and let $h \in\mathcal{M}(\posRealNum, \complNum^{d_Y \times d_U})$ be the inverse Laplace transform of $\mathbf{G}$ that exists by Proposition~\ref{lem:transferFunctLaplaceMeasure}. Let furthermore $(u,x,y)$ be a generalised solution in the distributional sense with $x(0) = 0$ and $u \in L^\infty_\loc(\posRealNum, U)$. Then $y = u \ast h$, $y \in L^\infty_\loc(\posRealNum, Y)$ and
\begin{equation}
\label{eq:distSolutonsBIBOIneq}
    \| y \|_{\lpSpaceDT{\infty}{[0,T]}{Y}} \leq \|h\|_{\mathcal{M}} \| u \|_{\lpSpaceDT{\infty}{[0,T]}{U}} \qquad \textrm{for all } T > 0.
\end{equation}
\end{proposition}
\begin{proof}
Let $u \in L^\infty_\loc(\posRealNum, U)$ and let $(u,x,y)$ be the corresponding generalised solution in the distributional sense with $x(0) = 0$.

Then consider for any $T > 0$ the input function $u_T := u|_{[0,T]} \in \lpSpaceDT{\infty}{\posRealNum}{U}$ giving rise to generalised solutions $(u_T, x_T, y_T)$, which by causality agree on $[0,T]$ with $(u,x,y)$.
Define furthermore
\begin{equation}
M_T(t) := \int_0^t u_T(s) \intd{s}, \qquad
        \begin{bmatrix} K_T(t) \\L_T(t) \end{bmatrix} :=  \int_0^t (t - s) \begin{bmatrix} x_T(s) \\u_T(s) \end{bmatrix} \intd{s}.
\end{equation}
Then $M_T$ and $L_T$ are locally absolutely continuous as Lebesgue integrals. In addition, $\|M_T\|_{\lpSpaceDT{\infty}{\posRealNum}{U}} \leq T \| u_T \|_{\lpSpaceDT{\infty}{\posRealNum}{U}}$ and $\|L_T\|_{\lpSpaceDT{\infty}{\posRealNum}{U}} \leq T^2 \| u_T \|_{\lpSpaceDT{\infty}{\posRealNum}{U}}$, so both functions are bounded.

By \cite[Lem.~4.7.9]{b_Staffans2005} $K_T \in C^1 (\posRealNum, X)$ and solves 
\begin{equation}
    \Dot{K}_T(t) = A K_T(t) + B L_T(t),
\end{equation}
for almost all $t > 0$ and, by \cite[Thm.~4.7.11]{b_Staffans2005} due to $L_T$ being continuous, even for all $t>0$. Thus $(L_T, K_T, (C \& D) \left[ \begin{smallmatrix} K_T(t) \\L_T(t) \end{smallmatrix} \right])$ is a classical solution of the system node.

As the system node is $C^\infty$-BIBO stable, we have $(C \& D) \left[ \begin{smallmatrix} K_T(t) \\L_T(t) \end{smallmatrix} \right] = h \ast L_T$ by Proposition~\ref{thm:BIBOInequalityAllClassSols}. Thus we find
\begin{equation}
    y_T = \frac{\textrm{d}^2}{\textrm{d}t^2} \left( (C\& D) \begin{bmatrix} K_T(t) \\L_T(t) \end{bmatrix} \right) =  \frac{\textrm{d}^2}{\textrm{d}t^2} \left( h \ast L_T \right).
\end{equation}
But then, by \cite[Thm.~3.7.1]{b_GripenbergLondenStaffans1990}, we find that $y_T =  h \ast \frac{\textrm{d}^2}{\textrm{d}t^2} L_T = h \ast u_T $ and thus, as this holds for any $T > 0$ and by causality we have $y = h \ast u$.
Now, as $h \in\mathcal{M}(\posRealNum, \complNum^{d_Y \times d_U})$, if $u \in L^\infty_\loc([0,\infty),U)$ then $y = u \ast h \in L^\infty_\loc([0,\infty),Y)$ by \cite[Cor.~3.6.2(i)]{b_GripenbergLondenStaffans1990}  and Equation~\eqref{eq:distSolutonsBIBOIneq} follows from \cite[Thm.~3.6.1(i)]{b_GripenbergLondenStaffans1990}.
\end{proof}

\begin{remark}
Note that the previous proposition does not pose a contradiction to examples such as the ones provided in \cite[Ex.~2.5.9]{b_Grafakos2014}, \cite[Rem.~3.8]{a_WeissRepresentationShiftInvariant1991} or \cite[Thm.~3.1]{a_PartingtonRepresentationShiftInvariant1998} of bounded, shift-invariant operators from $L^\infty$ to $L^\infty$ that do not admit a representation in terms of a convolution or a transfer function on all of $L^\infty$. Instead, the result here merely shows that these input-output operators cannot arise from or be represented in the form of a system node. 
\end{remark}

\subsection{Relation to the results from \cite{a_UnserNoteOnBiboStability}}
\label{sec:RelationToUngerResult}
At a cursory glance, the first central results of this contribution -- the characterisation of $C^\infty$-BIBO stability by the impulse response being a measure of bounded total variation in Theorem~\ref{thm:classicalBIBOBoundedVariation} and the equivalence to $L^\infty$-BIBO stability in Theorem~\ref{thm:generalBIBOclassicalBIBO} -- may seem to capture the same statement as \cite{a_UnserNoteOnBiboStability} in particular Theorem 4 therein. On a closer look there is however a major distinction to be made out.

What \cite{a_UnserNoteOnBiboStability} considers is a continuous operator $T_h: \mathcal{D}(\realNum) \rightarrow \mathcal{D}'(\realNum)$ defined on the space of test functions $\mathcal{D}(\realNum)$ in terms of a convolution $T_h(f) = f \ast h $ with a distribution $h$ and the question for which type of $h$ it admits a continuous extension in particular to an operator $T_h: L^\infty \rightarrow L^\infty$. It finds that this is the case if and only if $h \in \mathcal{M}(\posRealNum)$ is a measure of bounded total variation.

In contrast, this contribution studies the question under which circumstances the input-output operator of a system node is well-defined and bounded as an operator from $L^\infty$ to $L^\infty$. The way to showing that this again reduces to the impulse response existing and being a measure of bounded total variation leads -- similar to \cite{a_UnserNoteOnBiboStability} -- to first considering the restricted input-output operator acting only on the test functions as input, and using it to derive existence and boundedness of the full input-output operator.

While this may seem like just a direct application of Theorem 4 from \cite{a_UnserNoteOnBiboStability}, this is not the case. This result does show that a continuous extension of the restricted output operator exists and is given by the convolution with the distribution defining its behaviour on the test function. However, it is not at all clear that the thus constructed operator from $L^\infty$ to $L^\infty$ is actually the full input-output operator of the system. After all, it is not even clear that the output for any $L^\infty$ input is again in $L^\infty$. And furthermore by \cite[Ex.~2.5.9]{b_Grafakos2014} uniqueness of the extension is potentially not even given if boundedness of the full input-output operator was assumed.

Thus another way of understanding the main result from the first part of this contribution is showing that the input-output behaviour of the system node is first of all well-defined as an operator from $L^\infty$ to $L^\infty$ and that it is precisely the one given by the convolution extension as studied in \cite{a_UnserNoteOnBiboStability}.

\section{Sufficient conditions for BIBO stability}
\label{sec:SufficientConditions}
We start by giving sufficient (but not necessary) conditions for BIBO stability of particular types of system nodes. 

\subsection{Riesz-spectral systems}
\label{sec:RieszSpectralChapter}
We consider special system nodes $\Sigma(A,B,C, \mathbf{G})$, where 
$A: \dom(A) \rightarrow X$ is supposed to be a Riesz-spectral operator on a Hilbert space $X$ \cite{b_TucsnakWeiss2009} and $B:\complNum \rightarrow X_{-1}$ and $C: X_1 \rightarrow \complNum$ are scalar control and observation operators. This means that there exists a Riesz-basis $\set{\phi_n}_{n \in \natNum}$ in $X$ such that each $\phi_n$ is an eigenvector of $A$ with eigenvalue $\lambda_n$. By $\set{\psi_n}_{n \in \natNum}$ we denote the associated biorthogonal sequence satisfying $\innerProd{\phi_n}{\psi_m} = \delta_{n,m}$ for all $m,n \in \natNum$. Furthemore, there exist sequences $(b_n)_{n \in \natNum}$ and $(c_n)_{n \in \natNum}$ given by $b_n := \innerProd{Bu}{\psi_n}$ and $c_n := C \phi_n$ respectively such that the operators $B$ and $C$ are given by $ B u = u \cdot b$ and $C x = \innerProd{x}{c}$ for all $u \in U$ and $x \in X_1$, respectively, with elements $b := \sum_k b_k \phi_k \in X_{-1}$ and $c := \sum_k c_k \psi_k \in X_{-1}$.
Such a system node is called a \emph{Riesz-spectral system}.

\label{sec:BIBORiesz:SufficientCondition}
Using the sequences  $\lambda_n$, $b_n$ and $c_n$ we can now give a sufficient condition for BIBO stability of a corresponding Riesz-spectral system node with scalar input and output.
\begin{proposition}
\label{prop:RieszSpectralSufficient}
Let $\Sigma(A,B,C,\mathbf{G})$ be a Riesz-spectral system with $A$ having eigenvalues only in 
the open left half-plane. Assume further that the coefficient sequences $(b_n)_{n \in \natNum}$ and $(c_n)_{n \in \natNum}$ satisfy 
\begin{equation}\label{condRiesz}
    \sum_k \left| \frac{b_k c^\ast_k}{\realPart{\lambda_k}} \right| < \infty.
\end{equation}
Then  $\Sigma(A,B,C,\mathbf{G})$ is $L^\infty$-BIBO stable.
\end{proposition}
\begin{proof}
As $\realPart{\lambda_n} < 0$ for all $n \in \natNum$,  we get for $s \in \overline{\complNum_0}$ that $\left| s - \lambda_n\right| \geq \left| \realPart{\lambda_n} \right| $
so that $ \left| \frac{b_n c_n^\ast}{s - \lambda_n} \right| \leq \left| \frac{b_n c_n^\ast}{\realPart{\lambda_n}} \right|$.
Since $\resolvent{s}{A} x = \sum_{n \in \natNum} \frac{\innerProd{x}{\psi_n} \phi_n}{s - \lambda_n}$ for $s \in \complNum_{\omega(\mathbb{T})}$, \cite[Prop. 2.6.2.]{b_TucsnakWeiss2009}, we find that the transfer function is given by 
\begin{equation}
\label{eq:RieszSpectral:GenTransferFunction}
    \mathbf{G}(s) = \alpha + \sum_{n \in \natNum} \frac{b_n c_n^\ast}{s - \lambda_n}.
\end{equation}
with $\alpha$ some constant. 
By the assumption and Fubini's theorem, 
\begin{equation}
\begin{split}
    \int_0^\infty \left| \sum_{n \in \natNum} b_n c_n^\ast e^{\lambda_n t} \right| \intd{t} &\leq \int_0^\infty  \sum_{n \in \natNum} \left| b_n c_n^\ast e^{\lambda_n t} \right|  \intd{t}
    = \int_0^\infty  \sum_{n \in \natNum} \left| b_n c_n^\ast \right| e^{\realPart{\lambda_n} t}   \intd{t} \\
    &\leq  \sum_{n \in \natNum} \int_0^\infty \left| b_n c_n^\ast \right| e^{\realPart{\lambda_n} t}   \intd{t} =  \sum_{n \in \natNum} \frac{\left| b_n c_n^\ast \right|}{| \realPart{\lambda_n} |} < \infty.
\end{split}
\end{equation}
Thus, by continuity of the Laplace transform, $\mathbf{G}=\laplaceTr{h}$ for
\begin{equation}
    h(t) = \alpha \, \delta(t) + \sum_{n \in \natNum} b_n c_n^\ast e^{\lambda_n t}.
\end{equation}
Since $h \in \mathcal{M}(\posRealNum, \complNum)$,  $\Sigma(A,B,C,\mathbf{G})$ is $L^\infty$-BIBO stable by Proposition~\ref{thm:classicalBIBOBoundedVariation}.
\end{proof}

\begin{remark}
\label{rem:RiesSufficientExtension}
\begin{enumerate}
    \item 

Proposition~\ref{prop:RieszSpectralSufficient} easily extends to the case of finitely many $\lambda_k$'s with $\realPart{\lambda_k} \ge 0$. The sufficient condition in this case becomes
\begin{equation}
    \sum_{\substack{k \in \natNum \\ \realPart{\lambda_k} < 0}} \left| \frac{b_k c^\ast_k}{\realPart{\lambda_k}} \right| < \infty \quad \andMath \quad \left(\realPart{\lambda_k} \ge 0 \Rightarrow b_k c^\ast_k = 0\right).
\end{equation}

\item It is easy to see that the condition in the proposition is not necessary. Indeed, let $\Sigma\left( \widetilde{A}, \widetilde{B}, \widetilde{C}, \widetilde{\mathbf{G}} \right)$ be the system defined by $(\widetilde{\lambda}_k) = (-1, -2, -3, \ldots )$, $(\widetilde{b}_k) = (1, 1, \ldots )$ and $(\widetilde{c}_k) = (1, 1, \ldots )$ and $\widetilde{\mathbf{G}}$ determined up to an additve constant. Then \eqref{condRiesz} is not satisfied. Now consider the ``stacked system'' $\Sigma\left( A, B, C, \mathbf{G} \right) = \left( \left[ \begin{smallmatrix}
    \widetilde{A} & 0 \\ 0 & \widetilde{A}
\end{smallmatrix} \right], \left[ \begin{smallmatrix}
    \widetilde{B} \\ \widetilde{B}
\end{smallmatrix} \right], \left[ \begin{smallmatrix}
    \widetilde{C} & - \widetilde{C}
\end{smallmatrix} \right], 0 \right)$, which is still Riesz-spectral and does not satisfy \eqref{condRiesz} either by construction. However, $\Sigma\left( A, B, C, \mathbf{G} \right)$ is $L^\infty$-BIBO stable as its transfer function is 0. For more detailed calculations we refer the reader to the proof of Theorem~\ref{thm:MultPertResult} below.

\item Even for Riesz-spectral systems, $L^\infty$-well-posedness is  strictly stronger than $L^\infty$-BIBO stability and we can charaterize this stronger property in the case that  $A$  generates an analytic, exponentially stable semigroup.
Then a Riesz-spectral system $\Sigma(A,B,C,\mathbf{G})$ is $L^{\infty}$-well-posed in the sense of \cite{b_Staffans2005} if and only if $C$ is bounded, i.e., $(c_n)_{n \in \natNum} \in \ell^2\left(\natNum, \complNum\right)$. The sufficiency holds by Lemma \ref{lem:ABIsystemBIBO} below and since $B$ is automatically infinite-time $L^{\infty}$-admissible, see \cite{a_JacobEtAlISS2018}. The necessity holds for general system nodes, \cite[Thm.~4.4.2]{b_Staffans2005}.  
\end{enumerate}
\end{remark}

\subsection{Integrability of $C \semiGroup{\cdot} B$}
In the case that the infinite-dimensional analogue of the impulse response $h(t) := C e^{t A} B$ of a finite-dimensional state-space system is well-defined, its integrability provides a sufficient condition for BIBO stability.

\begin{proposition}
\label{prop:CTBinL1BIBO}
Let $\Sigma(A,B,C,\mathbf{G})$ be a system node. Suppose  that there exists a Banach space $Z$ continuoulsy embedded in $X_{-1}$ 
such that 
\begin{itemize}
    \item $C$ has a continuous extension $C_{Z}$ to $Z$ 
    \item $\semiGroup{t}B\in Z$ for almost all $t>0$, and
    \item $C_{Z}T(\cdot)B\in \lpSpaceDT{1}{\posRealNum}{\boundedOp{U}{Y}}.$
\end{itemize}
Then $\Sigma(A,B,C,\mathbf{G})$ and $\Sigma(A^*,C^*,B^*,\mathbf{G}^*)$ are both $L^\infty$-BIBO stable.

\end{proposition}
\begin{proof}
If $C_{Z} \semiGroup{\cdot} B \in \lpSpaceDT{1}{\posRealNum}{\boundedOp{U}{Y}}$, then in particular its Laplace transform
\begin{equation}
    \mathbf{H}(s) = \int_0^\infty C_{Z} \semiGroup{t} B e^{s t} \intd{t}, \qquad s \in \complNum_\alpha
\end{equation}
exists on some right half-plane $\complNum_\alpha$.
At the same time we have for any $x \in X_{-1}$ that
\begin{equation}
    \resolvent{s}{A} \resolvent{r}{A} x = \frac{1}{r - s} \int_0^\infty \semiGroup{\sigma} \left( e^{-s \sigma} - e^{-r \sigma} \right) x \intd{\sigma},
\end{equation}
and thus for any $u \in U$ that $(r - s) C \resolvent{s}{A} \resolvent{r}{A} B u 
    = \left( \mathbf{H}(s) - \mathbf{H}(r) \right) u$.
But this implies that $\mathbf{H} = \mathbf{G} + c$ for some $c \in \complNum$ so that we find
\begin{equation}
    \laplaceInvTr{\mathbf{G}} = \laplaceInvTr{\mathbf{H}} - c \delta(\cdot) = C_{Z} \semiGroup{\cdot} B - c \delta (\cdot),
\end{equation}
which is a measure of bounded variation by the assumption. But then we can argue analogously to the proof of Theorem~\ref{thm:generalBIBOclassicalBIBO} to show that the system node is $L^\infty$-BIBO stable. More precisely, the properties of convolutions and Laplace transforms of matrix-valued measures employed therein have to be replaced by the corresponding results for general operator-valued $\lpSpaceDT{1}{\posRealNum}{\boundedOp{U}{Y}}$ functions (see e.g.\ \cite[Lem.~D.1.11]{th_Mikkola2002}).

The second statement follows from the observation that clearly also $B^* \semiGroup{\cdot}^* C^* \in \lpSpaceDT{1}{\posRealNum}{\boundedOp{Y^*}{U^*}}$ if $C \semiGroup{\cdot} B \in \lpSpaceDT{1}{\posRealNum}{\boundedOp{U}{Y}}$. 
\end{proof}

For analytic semigroups also the converse of this proposition holds.
\begin{proposition}
    Let $\Sigma(A,B,C,\mathbf{G})$ be a $C^\infty$-BIBO stable system node with $A$ the generator of an analytic semigroup and finite-dimensional input and output spaces $U$ and $Y$. Then the function $t \mapsto C \semiGroup{t} B$ is in $\lpSpaceDT{1}{\posRealNum}{\boundedOp{U}{Y}}$.
\end{proposition}
\begin{proof}
    It suffices to show this for $U = Y = \complNum$. Then we can identify $B$ and $C$ with elements $b \in X_{-1}$ and $c \in (X_1)^\ast$, such that $B u = b \, u$ for any $u \in U$ and $C x = \innerProd{x}{c}_{X_1,(X_1)^\ast}$ for any $x \in X_1$.

    For fixed $t > \frac{\epsilon}{2} > 0$ and any $u \in C^\infty_c\left( \epsilon, t - \epsilon \right)$ and extended with 0 to $\posRealNum$, by Lemma~\ref{lem:existenceClassicalSolutions}, there exists a classical solution $(u,x,y)$ with $x(0) = 0$. Furthermore, by the $C^\infty$-BIBO stability there is $k>0$ such that for all such solutions we have $\| y \|_{\lpSpaceDT{\infty}{\posRealNum}{\complNum}} \leq k \| u \|_{\lpSpaceDT{\infty}{\posRealNum}{\complNum}}$. Then we have
    \begin{align}
        | y(t) | &= \left| C \left( \int_0^t \semiGroup{t-s} B u(s) \intd{s} - \resolvent{\alpha}{A} B u(t) \right) + \mathbf{G}(\alpha) u(t) \right| \\
        &= \left| \int_\epsilon^{t-\epsilon}  \innerProd{\semiGroup{s} b}{c} u(t-s) \intd{s} \right| \leq k \| u \|_{\lpSpaceDT{\infty}{\posRealNum}{\complNum}}.
    \end{align}
    Now, as $\innerProd{\semiGroup{\cdot} b}{c}$ is in particular continuous on $[\epsilon, t-\epsilon]$, it is also in $\lpSpaceDT{1}{[\epsilon, t-\epsilon]}{\complNum}$. By an adapted version of \cite[Lem.~9.17]{b_HaaseFunctionalAnalysis2014} we then conclude that $\|\innerProd{\semiGroup{\cdot} b}{c}\|_{\lpSpaceDT{1}{[\epsilon, t-\epsilon]}{\complNum}} \leq k$ independent of $t$ and $\epsilon$. But this then implies first that $\|\innerProd{\semiGroup{\cdot} b}{c}\|_{\lpSpaceDT{1}{[0, t]}{\complNum}} \leq k$ for all $t > 0$ and thus also $\|\innerProd{\semiGroup{\cdot} b}{c}\|_{\lpSpaceDT{1}{\posRealNum}{\complNum}} \leq k$.
\end{proof}

\subsection{A related result for more general systems}
Using Propsition~\ref{prop:CTBinL1BIBO}, one finds that there are sufficient conditions for BIBO stability of systems that may not be of Riesz-spectral form, which are closely related to the one from Proposition~\ref{prop:RieszSpectralSufficient}.

\begin{theorem}
\label{thm:sufficientResultSmaller1}
Let $\Sigma(A,B,C,\mathbf{G})$ be a system node where $A$ generates an exponentially stable and analytic semigroup $\semiGroupDef$, $B \in \boundedOp{U}{X_{-\alpha}}$ and $C \in \boundedOp{X_{\beta}}{Y}$ with $\alpha + \beta < 1$. Then $\Sigma(A,B,C,\mathbf{G})$ is $L^\infty$-BIBO stable.
\end{theorem}
\begin{proof}
    As $B \in \boundedOp{U}{X_{-\alpha}}$ we have $\widetilde{B} := (-A)^{-\alpha} B \in \boundedOp{U}{X}$ and as $C \in \boundedOp{X_{\beta}}{\complNum}$ we have $\widetilde{C} := C (-A)^{-\beta} \in \boundedOp{X}{Y}$. Furthermore, due to the analyticity of the semigroup, $C \semiGroup{t} B$ is well-defined for all $t>0$.
    Then we have
    \begin{align}
        \| C \semiGroup{t} B \|_{\boundedOp{U}{Y}} &= \| \widetilde{C} (-A)^{\alpha + \beta}  \semiGroup{t - s} \widetilde{B} \|_{\boundedOp{U}{Y}} \\ &\leq \| \widetilde{C} \|_{\boundedOp{X}{Y}} \|\widetilde{B} \|_{\boundedOp{U}{X}} \| (-A)^{\alpha + \beta}  \semiGroup{t - s}  \|_{\boundedOpSelf{X}},
    \end{align}
    and thus by the estimate from \cite[Thm.~2.6.13]{b_Pazy1983} find that for all $t > 0$ we have
    \begin{equation}
        \int_0^t \| C \semiGroup{s} B \|_{\boundedOp{U}{Y}} \intd{s} \lesssim \int_0^\infty s^{-\alpha-\beta} e^{-\delta s} \intd{s},
    \end{equation}
    with constants independent of $t$. 
    Hence $C \semiGroup{\cdot} B \in \lpSpaceDT{1}{\posRealNum}{\boundedOp{U}{Y}}$. Then, by Proposition~\ref{prop:CTBinL1BIBO}, the system is $L^\infty$-BIBO stable.
\end{proof}

\begin{theorem}
\label{thm:sufficientResultEqual1}
Let $\Sigma(A,B,C,\mathbf{G})$ be a system node on a Hilbert space $X$, where $A$ generates an exponentially stable and analytic semigroup $\semiGroupDef$ that is similar to a contraction semigroup, $B \in \boundedOp{U}{X_{-\alpha}}$ and $C \in \boundedOp{X_{\beta}}{Y}$ with $\alpha + \beta = 1$ and finite-dimensional $U$ and $Y$. Then $\Sigma(A,B,C,\mathbf{G})$ is $L^\infty$-BIBO stable.
\end{theorem}
\begin{proof}
    It suffices to consider $U = Y = \complNum$. By the assumptions on $B$ and $C$ there exist $b,c \in X$ such that $B u = u \cdot (-A)^{\alpha} b$ and $C x = \innerProd{x}{(-A^*)^{\beta} c}$. Thus, using weak square function estimates \cite{a_CowlingEtAl96FunctionalCalculus}, by \cite[Thm.~1~\&~Lem.~12]{a_JacobSchwenningerZwartParabolicISS2019} we find that 
\begin{equation}
    \int_0^\infty  \left| C \semiGroup{s} B \right| \intd{s}  =  \int_0^\infty  \left| \innerProd{A \semiGroup{s} b}{ c} \right| \intd{s}  < \infty.
\end{equation}
Then again, by Proposition~\ref{prop:CTBinL1BIBO}, the system is $L^\infty$-BIBO stable.
\end{proof}

\section{BIBO stability under multiplicative perturbations}
\label{sec:MultPerturbationsChapter}
As a second question we consider the conservation of BIBO stability under multiplicative perturbations  $\Sigma(AP,B,C,\widetilde{\mathbf{G}})$ of the system node $\Sigma(A,B,C,\mathbf{G})$, where $P:X \rightarrow X$ is a bounded operator and $\widetilde{\mathbf{G}}$ a transfer function solving Equation~\eqref{eq:transferFunctionDifferenceRelation} for the triple $(AP, B,C)$. 
However, BIBO stability is not preserved in general, even if exponential stability is.

Recall that an operator $B \in \boundedOp{U}{X_{-1}}$ is called \emph{$L^p$-control-admissible} for a semigroup $\semiGroupDef$ if for all $t > 0$ the map $\Phi_t: \lpSpaceDT{p}{\posRealNum}{U} \rightarrow X_{-1}$ given by $\Phi_t(u) = \int_0^t \semiGroup{t-s} B u(s) \intd{s}$ maps into $X$ and is in $\boundedOp{\lpSpaceDT{p}{\posRealNum}{U}}{X}$ \cite[Def.~4.1]{a_Weiss89ControlOperators}. An operator $C \in \boundedOp{X_1}{Y}$ is called \emph{$L^p$-observation-admissible} if for some $t>0$ and any $x \in X_1$ the function $C \semiGroup{\cdot} x: [0,t] \rightarrow Y$ is in $\lpSpaceDT{p}{[0,t]}{Y}$ and the map $\Psi_t: X_1 \rightarrow \lpSpaceDT{p}{[0,t]}{Y}$ given by $\Psi_t(x) = C \semiGroup{\cdot} x$ is bounded with respect to the $X$-norm \cite[Def.~6.1]{a_WeissObservationOperators1989}. Furthermore, admissible operators are called \emph{infinite-time admissible} if the families of maps $\Phi_t$ respectively $\Psi_t$ are bounded uniformly in t.

Furthermore, we remind the reader that a system node $\Sigma(A,B,C,\mathbf{G})$ is called \emph{$L^2$-well-posed} if $B$ and $C$ are $L^2$-control-admissible and $L^2$-observation-admissible and in addition $\mathbf{G}$ is analytic and bounded on some right half-plane \cite[Prop.~4.9]{a_TucsnakWeissWellPosed2014}.

\begin{theorem}
\label{thm:MultPertResult}
For every separable Hilbert space $X$ there exists an $L^2$-well-posed, exponentially stable and $L^\infty$-BIBO stable system node $\Sigma(A,B,C,\mathbf{G})$ 
and a bounded and coercive operator $P: X \rightarrow X$ such that any system node  $\Sigma(AP,B,C,\widetilde{\mathbf{G}})$
is exponentially stable, but neither $L^\infty$-BIBO stable nor $L^2$-well-posed.
\end{theorem}
\begin{proof}
Let $(\phi_k)$ be a  Riesz basis of the Hilbert space $X$ with biorthogonal sequence $(\psi_k)$. Then consider the sequence $(\lambda_k) = (-1, -1, -2, -2, \ldots ) = \begin{cases} - \frac{k+1}{2} &  k \textrm{ odd} \\ -\frac{k}{2}  &  k \textrm{ even}  \end{cases}$.
By \cite[Prop.~2.6.2~\&~2.6.5]{b_TucsnakWeiss2009} we can define the diagonalizable operator $A: \dom(A) \subset X \rightarrow X$ using the sequence $(\lambda_k)$ that generates a strongly continuous semigroup $\semiGroupDef$ on $X$ that is clearly exponentially stable.

Furthermore, let $(b_k)_{k \in \natNum}$ and $(c_k)_{k \in \natNum}$ be the sequences $(b_k) = (1, 1, \ldots )$ and $(c_k) = (1, -1, 1, -1, \ldots ) = \begin{cases} 1 &  k \textrm{ odd} \\ -1  &  k \textrm{ even}  \end{cases}$.
Both of these sequences satisfy the Carleson measure criterion with respect to the eigenvalue sequence $(\lambda_k)_{k \in \natNum}$. Thus we can define operators $B := \cdot \, b \in \boundedOp{\complNum}{X_{-1}}$ and $C := \innerProd{\cdot}{c} \in \boundedOp{X_1}{\complNum}$ using $b := \sum_k b_k \phi_k \in X_{-1}$ and $c := \sum_k c_k^* \phi_k \in X_{-1}$ \cite[Thm.~5.3.2]{b_TucsnakWeiss2009}. Moreover, $B$ and $C$ are infinite-time $L^2$-admissible observation and control operators for the semigroup $\mathbb{T}$, respectively.

Using that $(s-A)^{-1} x = \sum_{k} \frac{\innerProd{x}{\psi_k}}{s - \lambda_k} \phi_k$ for all $s \in \complNum_0$,
we find that Equation~\eqref{eq:transferFunctionDifferenceRelation} takes the form $\mathbf{G}(s) - \mathbf{G}(t) = (t - s) \sum_{k} s_k$ for $s,t \in \complNum_0$,
where
\begin{equation}
    s_k = \begin{cases} \frac{1}{\left(s + \frac{k+1}{2}\right)\left(t+\frac{k+1}{2}\right)}  &  k \textrm{ odd} \\ -\frac{1}{\left(s + \frac{k}{2}\right)\left(t+\frac{k}{2}\right)}  &  k \textrm{ even} \end{cases}.
\end{equation}
This sum converges absolutely and by reordering 
we find that $\sum_{k} s_k = 0$ and for $s \in \complNum_0$ thus any transfer function for $(A,B,C)$ satisfies $\mathbf{G}(s) = \alpha$ with $\alpha \in \complNum$. 

We conclude that  $\Sigma(A,B,C,\mathbf{G})$ is indeed $L^2$-well-posed as $\mathbb{G} \in H^{\infty}(\complNum_0)$ \cite[Prop. 4.9]{a_TucsnakWeissWellPosed2014}. Furthermore $\laplaceInvTr{G(s)}(t) = \alpha \, \delta(t)$ exists in the distributional sense and is a measure of bounded variation, showing that the system is also $L^\infty$-BIBO stable.

Consider now the sequence $(p_k) = (1, 2, 1, 2 \ldots) = \begin{cases} 1 &  k \textrm{ odd} \\ 2  &  k \textrm{ even}  \end{cases}$. Then, by \cite[Prop.~2.5.4]{b_TucsnakWeiss2009} this defines an operator $P \in \boundedOpSelf{X}$ by $P x = \sum_k p_k \innerProd{x}{\psi_k} \phi_k$,
that is furthermore coercive.
The operator $A P: \dom(A) \rightarrow X$ is then still diagonal with respect to the Riesz basis $\phi_k$, being represented by the sequence $(\widetilde{\lambda}_k) = (\lambda_k p_k)$, and generates a strongly continuous semigroup $\semiGroupDefAlt{\widetilde{\mathbb{T}}}$ that is again exponentially stable. Furthermore, $B$ and $C$ are still infinite-time $L^2$-admissible control and observation operators as the sequences $(b_k)$ and $(c_k)$ still satisfy the Carleson measure criterion.

Employing Expression~\eqref{eq:transferFunctionDifferenceRelation} we find that for any transfer function $\widetilde{\mathbf{G}}(s)$ of the perturbed system $\Sigma(AP, B,C, \widetilde{\mathbf{G}})$ we have that $\widetilde{\mathbf{G}}(s) - \widetilde{\mathbf{G}}(t) = (t - s) \sum_{k} \widetilde{s}_k$
with 
\begin{equation}
    \widetilde{s}_k = \begin{cases} \frac{1}{\left(s + \frac{k+1}{2}\right)\left(t+\frac{k+1}{2}\right)}  &  k \textrm{ odd} \\ -\frac{1}{\left(s + k\right)\left(t+k\right)}  &  k \textrm{ even} \end{cases}.
\end{equation}
This sum again converges absolutely and, after reordering, we can read off that
\begin{equation}
    \widetilde{\mathbf{G}}(s) = \frac{\psi\left( 1 + \frac{s}{2} \right)}{2} - \psi(1+s) + \widetilde{\alpha},
\end{equation}
with $\psi(s) = \der{}{s} \ln \Gamma(s)$ being the digamma function \cite[6.3.1]{b_AbramowitzStegun1972} and $\widetilde{\alpha}$ some constant.

We observe that $\lim_{s \rightarrow \infty} \widetilde{\mathbf{G}}(s) = - \infty$ and thus the system node $\Sigma(AP, B, C,\widetilde{\mathbf{G}})$ is not $L^2$-well-posed. Furthermore, by \cite[Thm.~3.8.2]{b_GripenbergLondenStaffans1990}, if the inverse Laplace transform of $\widetilde{\mathbf{G}}$ were a measure of bounded variation, then $\widetilde{\mathbf{G}}$ would be bounded on some right half-plane. Therefore the system cannot be $L^\infty$-BIBO stable. 
\end{proof}

The above relation between BIBO stability and $L^{2}$-well-posedness is not that accidental, as the following remark shows.
\begin{remark}
As the Laplace transform of a measure of bounded total variation, the transfer function $\mathbf{G}$ of a BIBO stable system node is in particular bounded and analytic on the open right half-plane $\complNum_0^+$ \cite[Thm. 3.8.2]{b_GripenbergLondenStaffans1990}. In the case that $B$ and $C$ are $L^2$-admissible control and observation operators respectively, $L^2$-well-posedness is thus implied by BIBO stability \cite[Prop. 4.9]{a_TucsnakWeissWellPosed2014}.
\end{remark}

\section{BIBO stability under additive generator perturbations}
\label{sec:AddPerturbationsChapter}
Finally, we want to consider the analogous question to the one discussed in Section~\ref{sec:MultPerturbationsChapter}, but now for a bounded \emph{additive} perturbation, i.e.\ going from a system node $\Sigma(A,B,C, \mathbf{G})$ that is known to be BIBO stable to $\Sigma(A + P,B,C, \widetilde{\mathbf{G}})$ with $P \in \boundedOpSelf{X}$.

\subsection{System nodes with the identity as control/observation operator}
We begin by considering two particular types of system nodes, namely those for which either $B$ or $C$ is the identity operator. Note that for these the input resp. output space will not be finite-dimensional, so that the results of Section~\ref{sec:BIBOStability} do not apply.

In particular, this means that we cannot establish BIBO stability using the condition on the inverse Laplace transform of the transfer function. Instead we will have to directly show that the BIBO-inequality~\eqref{eq:BIBOInequality} is satisfied.
\begin{lemma}
\label{lem:ABIsystemBIBO}
Let $\Sigma(A,B,\mathbb{I}, \mathbf{G})$ be a system node with $B$ infinite-time $L^\infty$-control-admissible. Then $\Sigma(A,B,\mathbb{I}, \mathbf{G})$ is $L^\infty$-BIBO stable. \newline
In particular, the assertion holds if $B$ is $L^p$-control-admissible for some $1 \leq p \leq \infty$ and  $A$ generates an exponentially stable semigroup. 
\end{lemma}
\begin{proof}
Let $u \in \lpSpacelocDT{\infty}{\posRealNum}{X}$ and $(u,x,y)$ be the corresponding distributional solution with $x(0) = 0$ of the system node. Then the infinite-time $L^\infty$-control-admissibility implies that there exists $c > 0$ such that $x(t) = \int_0^t \semiGroup{t - s} B u(s) \intd{s} \in X$ and $\left\| x(t) \right\|_X \leq c \left\| u \right\|_{\lpSpaceDT{\infty}{[0,t]}{X}}$ for all $t>0$,
see e.g.\ \cite{a_JacobSchwenningerZwartParabolicISS2019}. The function
\begin{equation}
    y(t) = C\&D \begin{bmatrix} x(t) \\ u(t) \end{bmatrix} = x(t) - \resolvent{\beta}{A} B u(t) + \mathbf{G}(\beta) u(t),
\end{equation}
is measurable and thus $y$ is a function and 
\begin{equation}
\begin{split}
    \left\| y(t) \right\|_X &= \left\| x(t) - \resolvent{\beta}{A} B u(t) + \mathbf{G}(\beta) u(t) \right\|_X \\
    &\leq \left\| x(t) \right\|_X + \| \resolvent{\beta}{A} \| \| B \| \left| u(t) \right|_U + \| \mathbf{G}(\beta) \| \left| u(t) \right|_U \\
    &\leq  \left( c + \| \resolvent{\beta}{A} \| \| B \| + \| \mathbf{G}(\beta) \| \right) \left\| u \right\|_{\lpSpaceDT{\infty}{[0,t]}{X}},
\end{split}
\end{equation}
showing that the system node is $L^\infty$-BIBO stable.\newline 
The second part of the lemma directly follows from the first by noting that $L^{p}$-admissibility implies infinite-time $L^{\infty}$-admissibility when the semigroup is exponentially stable, see e.g.\ \cite[Lem.~2.9]{a_JacobEtAlISS2018}.
\end{proof}

\begin{remark}
Note that the converse of Proposition~\ref{lem:ABIsystemBIBO}, i.e.\ that $L^\infty$-BIBO stability of the system node $\Sigma(A,B,\mathbb{I},\mathbf{G})$ implies $L^\infty$-control-admissibility of $B$, does not hold.
Indeed, following \cite[Example~2.3]{a_JacobSchwenningerWintermayr2022}, consider the system node $\Sigma(A,A_{-1},\mathbb{I})$ on the state space $X = c_0(\natNum)$ with the diagonal operator $A = \textrm{diag}(n)$. Then this system node is $L^\infty$-BIBO stable, but $A_{-1}$ is not $L^\infty$-control-admissible.
\end{remark}

\begin{corollary}
\label{cor:AICsystemBIBO}
Let $\Sigma(A,B,C, \mathbf{G})$ be a system node with  an $L^p$-observation-admissible operator $C \in \boundedOp{X_1}{Y}$ for some $p\in(1,\infty)$, $B \in \boundedOp{U}{X}$ and $A$ the generator of an exponentially stable and analytic semigroup. Then $\Sigma(A,B,C, \mathbf{G})$ is  $L^\infty$-BIBO stable.
\end{corollary}
\begin{proof}
This follows from Proposition~\ref{thm:sufficientResultSmaller1}, as any $L^p$-observation-admissible operator $C$ is bounded as an operator from $X_{\alpha}$ to $Y$ for some $\alpha\in(0,1)$, see e.g.\ \cite{a_PreusslerSchwenninger2023}.
\end{proof}

The following shows that Corollary~\ref{cor:AICsystemBIBO} does not hold for arbitrary semigroups. 
\begin{proposition}
\label{prop:counterExampleAIC}
Let $A$ generate a right-invertible strongly continuous semigroup $\semiGroupDef$ and let $C:X_1 \rightarrow Y$ be an observation operator such that there exists  $c>0$ such that for all classical solutions $(u,x,y)$ of $\Sigma(A, \mathbb{I}, C, C \resolvent{\cdot}{A})$ with $x(0) = 0$ and all $T > 0$ we have $ \| y \|_{L^\infty([0,T],Y)} \leq c \| u \|_{L^\infty([0,T],X)} $.
Then $C$ extends to a bounded operator from $X$ to $Y$.
\end{proposition}
\begin{proof}
Let $x_U \in \dom(A)$ and define the input function $u = \semiGroup{\cdot} x_U$. Then the function 
$t \mapsto x(t) := \int_0^t \semiGroup{t - s} u(s) = t \semiGroup{t} x_U$ satisfies $x(t) \in X_1$ for all $t \geq 0$ and is continuously differentiable in $X$ with derivative $\dot{x}(t) = \semiGroup{t} x_U + t A \semiGroup{t} x_U = A x(t) + u(t)$ \cite[Thm.~3.2.1]{b_Staffans2005}. Hence $(u,x,Cx)$ is a classical solution of the system node $\Sigma(A, \mathbb{I}, C, C \resolvent{\cdot}{A})$ and by the assumption we thus have that for $t > 0$
\begin{equation}
\begin{split}
    \left\| t C \semiGroup{t} x_U \right\|_Y 
    \leq \| C x(\cdot) \|_{\lpSpaceDT{\infty}{[0,t]}{Y}} &\lesssim \| u \|_{\lpSpaceDT{\infty}{[0,t]}{X}} 
    \leq  \sup_{\tau \in [0,t]} \left\| \semiGroup{\tau} \right\| \| x_U \|_X.
\end{split}
\end{equation}
Thus $t C \semiGroup{t}$ 
extends continuously to an operator in $\boundedOp{X}{Y}$.
But then we find for some fixed $t > 0$ and with $S(t) \in \boundedOpSelf{X}$ denoting the right inverse of $\semiGroup{t}$ that
\begin{equation}
    \left\| t \, C x \right\|_Y = \left\| t \, C \, \semiGroup{t} S(t) x \right\|_Y \leq \left\| t \, C \, \semiGroup{t}  \right\| \left\| S(t) x \right\|_X \leq \left\| t \, C \, \semiGroup{t}  \right\| \left\| S(t) \right\| \left\| x \right\|_X,
\end{equation}
and thus $\left\|C x \right\|_Y \lesssim \|x\|_X$ implying that $C$ extends to $C \in \boundedOp{X}{Y}$.
\end{proof}
Note furthermore that $C$ being $L^1$-observation-admissible is in general not sufficient to ensure $L^\infty$-BIBO stability of $\Sigma(A, \mathbb{I}, C, C \resolvent{\cdot}{A})$.
\begin{proposition}
\label{prop:counterexampleAICL1}
There are system nodes $\Sigma(A, \mathbb{I}, C, C \resolvent{\cdot}{A})$ with $C$ being $L^1$-observation-admissible that are not $L^\infty$-BIBO stable.
\end{proposition}
\begin{proof}
There exists a Hilbert space $X$ and an unbounded operator $A$ that generates an exponentially stable and analytic semigroup such that any operator $C \in \boundedOp{X_1}{\complNum}$ is $L^1$-observation-admissible \cite{a_JacobSchwenningerZwartParabolicISS2019}. Furthermore, any such $C$ can be written for all $x \in X_1$ as $C x = \innerProd{\widetilde{c}}{Ax}_{X}$
with some $\widetilde{c} \in X$.

Assume for any such $C$ the system node  $\Sigma(A, \mathbb{I}, C, C \resolvent{\cdot}{A})$ was $L^\infty$-BIBO stable. Then consider, for any $f \in \lpSpaceDT{\infty}{\posRealNum}{X}$ with $\|f\|_{\lpSpaceDT{\infty}{\posRealNum}{X}} = 1$ the map
\begin{equation}
    S_f: \boundedOp{X_1}{\complNum} \rightarrow \lpSpaceDT{\infty}{\posRealNum}{Y}, \qquad
    C \mapsto C \int_0^\cdot \semiGroup{\cdot - s} f(s) \intd{s}.
\end{equation}
One can show that each of these maps is closed and thus by a closed graph argument bounded. Then by the uniform boundedness principle, there exists $K > 0$ such that for any $C \in \boundedOp{X_1}{\complNum}$ and any $f \in \lpSpaceDT{\infty}{\posRealNum}{X}$
\begin{equation}
    \left\| C \int_0^\cdot \semiGroup{\cdot - s} f(s) \intd{s} \right\|_{\lpSpaceDT{\infty}{\posRealNum}{Y}} \leq K \left\| C \right\|_{\boundedOp{X_1}{\complNum}} \left\| f \right\|_{\lpSpaceDT{\infty}{\posRealNum}{X}}.
\end{equation}
Thus, as
\begin{equation}
    \begin{split}
        \left\| C \int_0^\cdot \semiGroup{\cdot - s} f(s) \intd{s} \right\|_{\lpSpaceDT{\infty}{\posRealNum}{Y}} 
        &= \left\| \int_0^\cdot \innerProd{A^* \mathbb{T}^*(\cdot - s) \widetilde{c}}{f(s)}  \intd{s} \right\|_{\lpSpaceDT{\infty}{\posRealNum}{Y}},
    \end{split}
\end{equation}
and $\left\| C \right\|_{\boundedOp{X_1}{\complNum}} \lesssim \|\widetilde{c}\|_X$ we find that
\begin{equation}
\label{eqn:cCounterexampleInequality}
    \left\| \int_0^\cdot \innerProd{A^* \mathbb{T}^*(\cdot - s) \widetilde{c}}{f(s)}  \intd{s} \right\|_{\lpSpaceDT{\infty}{\posRealNum}{Y}} \leq K \|\widetilde{c}\|_X \left\| f \right\|_{\lpSpaceDT{\infty}{\posRealNum}{X}}.
\end{equation}
But now for any $s \in [0,t]$ we have $y_s = \frac{A^* \mathbb{T}^*(t - s) \widetilde{c}}{\left\| A^* \mathbb{T}^*(t - s) \widetilde{c} \right\|} \in X$ with $\|y_s\|_X = 1$ such that $\innerProd{A^* \mathbb{T}^*(t - s) \widetilde{c}}{y_s} = \left\| A^* \mathbb{T}^*(t - s) \widetilde{c} \right\|$.
Consider then the function $f: [0,t] \rightarrow X$ given by $ s \mapsto f(s) = \frac{A^* \mathbb{T}^*(t - s) \widetilde{c}}{\left\| A^* \mathbb{T}^*(t - s) \widetilde{c} \right\|}$.
Clearly $f$ is continuous on $[0,t)$ and thus also measurable on this interval. Furthermore, we have $\| f \|_{\lpSpaceDT{\infty}{[0,t]}{X}} = 1$. Then by \eqref{eqn:cCounterexampleInequality} we have that $\int_0^t \left\| A^* \mathbb{T}^*(t - s) \widetilde{c} \right\| \intd{s} \leq K \|\widetilde{c}\|_X$,
that is for any $t > 0$ and any $\widetilde{c} \in X$ we have that
\begin{equation}
    \left\| A^* \mathbb{T}^*(\cdot) \widetilde{c} \right\|_{\lpSpaceDT{1}{[0,t]}{X}} \leq K \|\widetilde{c}\|_X,
\end{equation}
i.e.\ that $A^*$ is $L^1$-observation-admissible. But this implies that $A$ is $L^\infty$-control-admissible \cite{a_WeissObservationOperators1989} and thus by \cite[Thm.~2.9]{a_JacobSchwenningerWintermayr2022} $A$ is bounded, giving a contradiction. Therefore the assumption cannot hold.
\end{proof}

\begin{remark}
We note that Lemma~\ref{lem:ABIsystemBIBO} and Proposition~\ref{prop:counterExampleAIC} also show that the duality notions between admissible control and observation operators do not fully extend to BIBO stability.
Indeed, consider a non-$L^\infty$-BIBO stable system node $\Sigma(A,\mathbb{I},C, \mathbf{G})$ with $A$ the generator of an exponentially stable semigroup on a Hilbert space $X$ and an $L^1$-control-admissible $C:X_{1}\to \mathbb{C}$, guaranteed to exist by Proposition~\ref{prop:counterexampleAICL1}. Its dual system node  $\Sigma(A^*,C^*,\mathbb{I},\mathbf{G}^*)$ is $L^\infty$-BIBO stable as $C^*$ $L^\infty$-control-admissible, by \cite{a_WeissObservationOperators1989} and Corollary~\ref{lem:ABIsystemBIBO}.  As the dual system of $\Sigma(A^*,C^*,\mathbb{I},\mathbf{G}^*)$ is again (isomorphic to) $\Sigma(A,\mathbb{I},C, \mathbf{G})$, we conclude that BIBO stability is not preserved under duality.
\end{remark}

\subsection{Additive perturbations}
\begin{theorem}
\label{thm:addPertGeneral}
Let $\Sigma(A,B,C,\mathbf{G})$ be an $L^{\infty}$-BIBO stable system node  and let $P \in \boundedOpSelf{X}$. If both $\Sigma(A + P,B,\mathbb{I}, \resolvent{\cdot}{A + P} B)$ and $\Sigma(A, \mathbb{I}, C, C\resolvent{\cdot}{A})$ are $L^{\infty}$-BIBO stable, then any additively perturbed system node $\Sigma(A + P,B,C, \widetilde{\mathbf{G}})$ is $L^{\infty}$-BIBO stable.
\end{theorem}
\begin{proof}
Let $u \in \lpSpacelocDT{\infty}{\posRealNum}{U}$. Then we have generalized solutions $(u, x_A, y_A)$ of  $\Sigma(A,B,C,\mathbf{G})$ with $x_A(0) = 0$ and $(u,x_{A+P}, y_{A+P})$ of $\Sigma(A + P,B,C, \widetilde{\mathbf{G}})$ with $x_{A+P}(0) = 0$. 

First, as $x_{A+P}$ solves $\dot{x}_{A+P} = (A+P) x_{A+P} + B u$ in $X_{-1}$ for almost all $t>0$ and $x_A$ solves $\dot{x}_A = A x_A + B u$ in $X_{-1}$ for almost all $t > 0$, we see that $\widetilde{x} = x_{A+P} - x_A$ solves $\dot{\widetilde{x}} = A \widetilde{x} + P x_{A+P}$ in $X_{-1}$ for almost all $t > 0$. Thus there exists a generalized solution $(P x_{A+P}, \widetilde{x}, \widetilde{y})$ of the system node $\Sigma(A, \mathbb{I}, C, C\resolvent{\cdot}{A})$ with 
\begin{equation}
    \widetilde{y} = C \left[ \widetilde{x} - \resolvent{\alpha}{A} P x_{A+P} \right] +  C\resolvent{\alpha}{A} P x_{A+P} = C \widetilde{x},
\end{equation}
to be understood in a distributional sense (meaning states and inputs are to be read as twice integrated and the resulting expression is twice differentiated).

As $(u,x_{A+P}, x_{A+P})$ is a generalised solution of the $L^\infty$-BIBO stable system node $\Sigma(A + P,B,\mathbb{I}, \resolvent{\cdot}{A + P} B)$, there exists $c_1 > 0$ such that for all $u \in \lpSpacelocDT{\infty}{\posRealNum}{U}$ and all $T>0$
\begin{equation}
    \| x_{A+P} \|_{\lpSpaceDT{\infty}{[0,T]}{X}} \leq c_1 \| u \|_{\lpSpaceDT{\infty}{[0,T]}{U}},
\end{equation}
whence $P x_{A+P} \in \lpSpacelocDT{\infty}{\posRealNum}{X}$. As $\Sigma(A, \mathbb{I}, C, C\resolvent{\cdot}{A})$ is $L^\infty$-BIBO stable we  conclude that $\widetilde{y} \in \lpSpacelocDT{\infty}{\posRealNum}{Y}$ and that there exists $c_2 > 0$ such that 
\begin{equation}
\begin{split}
    \| \widetilde{y} \|_{\lpSpaceDT{\infty}{[0,T]}{Y}} \leq c_2 \| P x_{A+P} \|_{\lpSpaceDT{\infty}{[0,T]}{X}} &\leq c_2 \| P \| \| x_{A+P} \|_{\lpSpaceDT{\infty}{[0,T]}{X}} \\ &\leq c_1 c_2 \| P \| \| u \|_{\lpSpaceDT{\infty}{[0,T]}{U}}.
\end{split}
\end{equation}
Further, consider for any $\alpha \in \complNum_{\omega(\mathbb{T})} \cap  \complNum_{\omega(\widetilde{\mathbb{T}})}$ (with $\widetilde{\mathbb{T}}$ being the semigroup generated by $A+P$) the expression from Definition~\ref{def:distributionalSolution} for the outputs $y_A$ and $y_{A+P}$ and let $\left(C\&D\right)_{A}$ and $\left(C\&D\right)_{A+P}$ be the combined output/feedthrough operators of the system nodes $\Sigma(A,B,C,\mathbf{G})$ and $\Sigma(A + P,B,C, \widetilde{\mathbf{G}})$ respectively. Then we have in a distributional sense (again meaning all $x$ and $u$ are to be understood as twice integrated and a distributional double derivative applied in the end, see also Definition~\ref{def:distributionalSolution}) 
\begin{equation}
\begin{split}
    &y_{A+P} - y_A = \left(C\&D\right)_{A+P} \begin{bmatrix} x_{A+P} \\ u \end{bmatrix} - \left(C\&D\right)_{A} \begin{bmatrix} x_{A} \\ u \end{bmatrix}    \\
    &= C \left[ \left( x_{A+P} - x_{A} \right) - \left( \resolvent{\alpha}{(A+P)} - \resolvent{\alpha}{A} \right) B u \right] + \left( \widetilde{\mathbf{G}}(\alpha) - \mathbf{G}(\alpha) \right) u \\
    &= \widetilde{y} - C \resolvent{\alpha}{A} P \resolvent{\alpha}{(A+P)} B u + \left( \widetilde{\mathbf{G}}(\alpha) - \mathbf{G}(\alpha) \right) u.
\end{split}
\end{equation}
Now, as $C \resolvent{\alpha}{A} P \resolvent{\alpha}{(A+P)} B \in \boundedOp{U}{Y}$ and $\widetilde{\mathbf{G}}(\alpha) - \mathbf{G}(\alpha) \in \boundedOp{U}{Y}$ this actually holds as an equality of functions and we have
\begin{equation}
    y_{A+P} = y_A + \widetilde{y}  - C \resolvent{\alpha}{A} P \resolvent{\alpha}{(A+P)} B u + \left( \widetilde{\mathbf{G}}(\alpha) - \mathbf{G}(\alpha) \right) u. 
\end{equation}
The statement then follows from the $L^\infty$-BIBO stability of $\Sigma(A,B,C,\mathbf{G})$ as
\begin{align}
    &\left\| C \resolvent{\alpha}{A} P \resolvent{\alpha}{(A+P)} B \, u(t) \right\|_Y \\  & \hspace{0.75cm}\leq \left\| C  \resolvent{\alpha}{A} \right\|_{\boundedOp{X}{Y}} \left\| P \right\|_{\boundedOpSelf{X}} \left\| \resolvent{\alpha}{(A+P)}  B \right\|_{\boundedOp{U}{X}} \left\| u(t) \right\|_U.
\end{align}
\end{proof}
\begin{remark}
The construction carried out in the proof of Theorem~\ref{thm:addPertGeneral} can be understood as decomposing the perturbed system node $\Sigma(A+P,B,C,\widetilde{\mathbf{G}})$ into the unperturbed system node $\Sigma(A,B,C,\mathbf{G})$, a pure gain part and a concatenation of $P$, $\Sigma(A, B, \mathbb{I}, \resolvent{\cdot}{(A+P)} B)$ and $\Sigma(A, \mathbb{I}, C, C\resolvent{\cdot}{A})$ as depicted in Figure~\ref{fig:perturbedSystemDecomposedl}.
\end{remark}
Using the two results of the previous section we then have the following corollary.
\begin{corollary}
\label{thm:addPertAnalSG}
Let $\Sigma(A,B,C,\mathbf{G})$ be an $L^{\infty}$-BIBO stable system node with $A$ being the generator of an exponentially stable, analytic semigroup $\semiGroupDef$, $B \in \mathcal{L}(U,X_{-1})$ an $L^\infty$-admissible control operator for $\mathbb{T}$ and $C\in\mathcal{L}(X_1 ,Y)$ an $L^p$-admissible observation operator for $\mathbb{T}$ for some $1 < p \leq \infty$. 
Let furthermore $P \in \boundedOpSelf{X}$ be such that $A + P: \dom(A) \rightarrow X$ generates an exponentially stable semigroup $\widetilde{\mathbb{T}}$.
Then any additively perturbed system node $\Sigma(A + P,B,C, \widetilde{\mathbf{G}})$
is $L^{\infty}$-BIBO stable.
\end{corollary}
\begin{proof}
$\Sigma(A + P,B,\mathbb{I}, \resolvent{\cdot}{A + P} B)$ is $L^\infty$-BIBO stable by Lemma~\ref{lem:ABIsystemBIBO} and $\Sigma(A, \mathbb{I}, C, C\resolvent{\cdot}{A})$ is $L^\infty$-BIBO stable by Corollary~\ref{cor:AICsystemBIBO} and thus the statement follows directly from Theorem~\ref{thm:addPertGeneral}.
\end{proof}

\begin{figure}
    \centering
    \resizebox{\textwidth}{!}{
    \begin{tikzpicture}[boxednode/.style={rectangle,thick,draw,inner xsep = 3mm,inner ysep=3mm},roundnode/.style={circle,thick,draw,inner xsep = 0.3mm,inner ysep=0.3mm}]
\node[boxednode]      (P)                              {$P$};
\node[boxednode]      (AIC)             [right= of P]                  {$\Sigma\left( A, \mathbb{I}, C, C \resolvent{\cdot\,}{A} \right)$};
\node[boxednode]      (ABI)             [left= of P]                  {$\Sigma\left( A + P, B, \mathbb{I}, \resolvent{\cdot\,}{(A+P)} B \right)$};
\node[boxednode]      (ABC)             [above=2cm of P]                  {$\Sigma\left( A, B,C,\mathbf{G} \right)$};
\node[boxednode] (MiddleNode) [above=0.5cm of P] {$\widetilde{\mathbf{G}}(\alpha) - \mathbf{G}(\alpha) - C \resolvent{\alpha}{A} P \resolvent{\alpha}{(A+P)} B$};
\node[roundnode]      (Plus)             [right=2cm of MiddleNode]                  {$+$};
\node      (InputSplit)          [circle,fill,inner sep=1pt]   [ left=3.8cm of MiddleNode]                  {};
\node (InputOriginRight)    [left=0.5cm of InputSplit] {};
\node (OutputRight)    [right=0.5cm of Plus] {};
\node (InputOriginRight)    [left=1cm of InputSplit] {};
\node (OutputRight)    [right=1cm of Plus] {};
\node (Equivalent)    [left=0.1cm of InputOriginRight] {$\equiv$};
\node (OutputLeft)    [left=0.1cm of Equivalent] {};
\node[boxednode]      (APBC)             [left=1cm of OutputLeft]                  {$\Sigma\left( A+P, B,C, \widetilde{\mathbf{G}} \right)$};
\node (InputOriginLeft)    [left=1cm of APBC] {};

\draw[-latex,thick] (ABI) -- (P);
\draw[-latex,thick] (P) -- (AIC);
\draw[-latex,thick] (AIC) -| (Plus.south);
\draw[-latex,thick] (ABC) -| (Plus.north);
\draw[-,thick] (ABI) -| (InputSplit.north);
\draw[-,thick] (ABC) -| (InputSplit.south);
\draw[-,thick] (InputOriginRight) -- (InputSplit.west);
\draw[-latex,thick] (Plus) -- (OutputRight);
\draw[-,thick] (InputOriginLeft) -- (APBC.west);
\draw[-latex,thick] (APBC.east) -- (OutputLeft);
\draw[-latex,thick] (MiddleNode) -- (Plus);
\draw[-,thick] (InputSplit) -- (MiddleNode);

\end{tikzpicture}
}
    \caption{Decomposition of the additively perturbed system}
    \label{fig:perturbedSystemDecomposedl}
\end{figure}
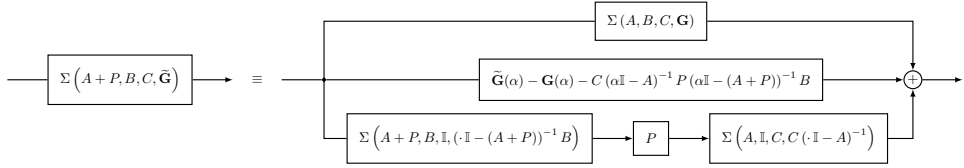

\section{Final remarks}
For an application of the results derived in this work to a concrete example, we refer the reader to \cite{a_HastirEtAl23NonlinearBIBO}, where a chemical reactor system has been studied in the context of Funnel control and BIBO stability of semilinear systems.

Within the study of BIBO stability for linear systems some open questions remain. First, the situation in the case of infinite-dimensional input or output spaces is, apart from a few special cases, unresolved. This includes the question of the relation between the two different BIBO notions and whether or not the inverse Laplace transform of the transfer function being an (operator-valued) measure of bounded total variation, is still is a necessary condition for BIBO stability.

Secondly, it would be desirable to extend perturbation results as in Section \ref{sec:AddPerturbationsChapter} to non-analytic semigroups as they appear e.g.\ in the study of hyperbolic PDEs.

\section*{Acknowledgments}
The authors would like to thank the anonymous reviewers for their very helpful comments.

\end{document}